\documentclass[11pt,twoside]{article}
\usepackage{amsmath,amssymb,epsfig,subfigure,graphicx,epstopdf,color,url}
\usepackage{multirow}
\usepackage{algorithm}
\usepackage{algorithmic}
\usepackage[titletoc,title]{appendix}
\newtheorem{proposition}{Proposition}[section]
\newtheorem{theorem}[proposition]{Theorem}
\newtheorem{lemma}[proposition]{Lemma}

\newtheorem{definition}[proposition]{Definition}
\newtheorem{remark}[proposition]{Remark}
\newtheorem{example}[proposition]{Example}

\newenvironment{proof}{{\noindent \em Proof.}}{\hfill $\fbox{}$ \vspace*{5mm}}
\numberwithin{equation}{section}

\newcommand{\ba}{{\bf a}}
\newcommand{\bb}{{\bf b}}

\newcommand{\bg}{{\bf g}}
\newcommand{\bh}{{\bf h}}
\newcommand{\bp}{{\bf p}}
\newcommand{\bq}{{\bf q}}
\newcommand{\br}{{\bf r}}

\newcommand{\bx}{{\bf x}}
\newcommand{\by}{{\bf y}}
\newcommand{\bz}{{\bf z}}
\newcommand{\bu}{{\bf u}}
\newcommand{\bv}{{\bf v}}
\newcommand{\bw}{{\bf w}}

\newcommand{\cd}{{\cal D}}
\newcommand{\cg}{{\cal G}}

\newcommand{\ck}{{\cal K}}
\newcommand{\cl}{{\cal L}}
\newcommand{\cm}{{\cal M}}

\newcommand{\cp}{{\cal P}}
\newcommand{\cq}{{\cal Q}}
\newcommand{\cs}{{\cal S}}

\newcommand{\la}{\lambda}

\newcommand{\R}{{\mathbb{R}}}
\newcommand{\Rm}{{\mathbb{R}^m}}
\newcommand{\Rn}{{\mathbb{R}^n}}
\newcommand{\Rr}{{\mathbb{R}^r}}
\newcommand{\Rmm}{{\mathbb{R}^{m\times m}}}
\newcommand{\Rnn}{{\mathbb{R}^{n\times n}}}

\newcommand{\Rmn}{{\mathbb{R}^{m\times n}}}
\newcommand{\Rmr}{{\mathbb{R}^{m\times r}}}
\newcommand{\Rrn}{{\mathbb{R}^{r\times n}}}

\newcommand{\diag}{{\rm diag}}
\newcommand{\dist}{{\rm dist}}

\newcommand{\card}{{\rm card}}

\newcommand{\dom}{{\rm dom}}

\newcommand{\tr}{{\rm tr}}

\newcommand{\BE}{\begin{equation}}
\newcommand{\EE}{\end{equation}}

\DeclareMathOperator*{\argmin}{argmin}

\begin{document}

\title{\bf A Column-Wise Update Algorithm for Sparse Stochastic Matrix Factorization}
\author{Guiyun Xiao\thanks{School of Mathematical Sciences, Xiamen University, Xiamen 361005, People's Republic of China  (xiaogy999@163.com).}
\and Zheng-Jian Bai\thanks{Corresponding author. School of Mathematical Sciences and Fujian Provincial Key Laboratory on Mathematical Modeling \& High Performance Scientific Computing,  Xiamen University, Xiamen 361005, People's Republic of China (zjbai@xmu.edu.cn). The research of this author was partially supported by the National Natural Science Foundation of China (No. 11671337) and the Natural Science Foundation of Fujian Province of China (No. 2021J01033).}
\and Wai-Ki Ching\thanks{Advanced Modeling and Applied Computing Laboratory, Department of Mathematics, The University of Hong Kong, Pokfulam Road, Hong Kong (wching@hku.hk). Research supported in part by Hong Kong RGC GRF Grant no. 17301519, IMR and Seed Funding for Basic Research from Faculty of Science, The University of Hong Kong.}}
\date{}
\maketitle
\begin{abstract}
Nonnegative matrix factorization arises widely in machine learning and data analysis.
In this paper, for a given factorization of rank $r$, we consider the sparse stochastic matrix factorization (SSMF) of decomposing a prescribed $m$-by-$n$ stochastic matrix $V$ into a product of an $m$-by-$r$ stochastic matrix $W$ and an $r$-by-$n$ stochastic matrix $H$, where both $W$ and $H$ are required to be sparse. With the prescribed sparsity level, we reformulate the SSMF as an unconstrained nonconvex-nonsmooth minimization problem and introduce a column-wise update algorithm for solving the minimization problem.
We show that our algorithm converges globally. The main advantage of our  algorithm is that the generated sequence converges to a special critical point of the cost function, which is nearly a global minimizer over  each column vector of the $W$-factor and is a global minimizer over the $H$-factor as a whole if there is no sparsity requirement on $H$.
Numerical experiments on both synthetic and real data sets are given to demonstrate the effectiveness of our proposed algorithm.
\end{abstract}

\vspace{3mm}
{\bf Keywords.}
Nonnegative matrix factorization, stochastic matrix factorization, sparsity,  alternating minimization, proximal gradient method

\vspace{3mm}
{\bf AMS subject classifications.} 65K05, 90C06, 90C26

\section{Introduction}\label{sec1}
Since the introduction of simple and efficient algorithms in the seminal work of Lee and Seung \cite{LS99},  nonnegative matrix factorization (NMF) has been widely adopted in various fields such as document clustering \cite{BG09,LC17}, computer vision \cite{BG07},
recommendation systems \cite{ZW06}, bioinformatics \cite{TB10}, face recognition \cite{KZ07,ZT06,ZF11}, acoustic signal processing \cite{YZ15}, source separation \cite{CZ09,GW14}, and
modeling default data via interactive hidden Markov model (IHMM) \cite{CF07,CS09}.

In this paper, we consider the following sparse stochastic matrix factorization (SSMF) problem.
Let  $V=(v_{ij}) \in \Rmn$ be a (column) stochastic data matrix, i.e., all its entries are nonnegative with each column summing to $1$ ($\sum_{i=1}^{m}v_{ij}=1$ for $j=1,\ldots,n$).
For a predetermined factor rank $r<\min\{m,n\}$, the SSMF attempts to find two  stochastic matrices $W\in\Rmr$ and $H\in\Rrn$ such that $V\approx WH$, where the sparseness constraints are imposed on both $W$ and $H$ (on columns).

In fact, the SSMF is  a sparse {\em probabilistic latent semantic analysis} (PLSA). PLSA  is a statistical technique for factor analysis of two-mode and co-occurrence data \cite{H99}. Compared to {\em latent semantic analysis}, PLSA possesses a solid
statistical foundation in model fitting, model selection and complexity control \cite{H01}.
Like NMF \cite{DL08}, PLSA has been used in many applications such as information retrieval and filtering, natural language processing, machine learning from text \cite{H99,H01}. In \cite{A16,AG12}, some conditions were presented for the existence of a unique solution of the stochastic matrix factorization. However, as noted in \cite{NL18}, these conditions may not be true  in the context of choice modeling.

There exist different methods for the sparse NMF since sparseness can achieve a sparse representation or data clustering.
In particular,  in \cite{H02}, Hoyer presented the projected gradient based multiplicative algorithm for solving the following minimization problem:
\BE\label{pro:smf-l1}
\min\limits_{W,H\ge 0}  C(W,H)=\frac{1}{2}\|V-WH\|_F^2+\la\|H\|_1,
\EE
where $V\ge 0$ is the input data matrix, $\la>0$ is a regularization parameter, and $\|H\|_1:=$$\sum_{i,j}|h_{ij}|$.
However, for the SSMF, the $\ell_1$-norm of $H$ is constant as $H$ is a stochastic matrix. There exist other cost functions for measuring the factorization residual such as the generalized {\em Kullback–Leibler} divergence \cite{DL08,LS01} and Minkowski family of metrics, see for instance \cite{WZ13} and the references therein. For additive regularization for stochastic matrix factorization, one may refer to \cite{VP14}. For other numerical algorithms for sparse topic modeling under the PLSA model, one may refer to \cite{WZ22} and the references therein. Other regularizers include the entropic prior \cite{SR07},  the pseudo-Dirichlet prior \cite{LU11}, and the $\ell_{1/2}$-regularization \cite{QJ11}.

To our best knowledge, there exist only a few works on the sparse NMF with $\ell_0$ constraints.
In particular, in \cite{XF21}, Xiu et al. gave a structured joint sparse NMF model:
\BE\label{pro:smf-sjs}
\begin{array}{ll}
\min\limits_{W,H} & \displaystyle  C(W,H)=\frac{1}{2}\|V-WH\|_F^2+\la\tr(H^TLH)\\[2mm]
\mbox{subject to (s.t.)} & W\ge 0,\quad  H\ge 0,\quad \|H\|_{2,0}\le s,
\end{array}
\EE
where $L$ is the graph Laplacian matrix learned from the input data matrix $V\ge 0$, $s$ is the prescribed  sparsity level, and $\|H\|_{2,0}$ denotes the $\ell_{2,0}$ norm of $H$, i.e., $\|H\|_{2,0}:=\card(\{t : \|H(t,:)\|\neq 0\})$. Also, an optimization algorithm based on the alternating direction method of multipliers was presented for solving the above model. However, it is not easy to choose an appropriate regularization parameter such that  a good  tradeoff between the residual term and the regularization term can be obtained \cite{LY14,WB04}.

Recently, there have been some development in cardinality/$\ell_0$-constrained optimization problems. In particular,
several optimization methods were introduced for cardinality-constrained problems appeared in portfolio optimization and statistical learning \cite{BS09,BB18,DA19}.
In \cite{XD19,XL16}, Xu et al. presented some projected gradient methods for cardinality constrained optimization appeared in compressed sensing, financial optimization and image processing. In \cite{KR21}, Kanzow et al. gave an augmented Lagrangian method for cardinality-constrained optimization problems.

In this paper, we directly apply $\ell_0$ constraints to measure the sparseness of the $W$-factor and the $H$-factor in the SSMF. 
We find a solution to the SSMF by minimizing the distance between the input stochastic data matrix $V$ and the product $WH$ in Frobineus norm, where each column of the stochastic matrices  $W$ and $H$ has the prescribed sparsity level.
To our best knowledge, most optimization algorithms for solving the sparse NMF treat each of the two factors as a whole, which may lead to slow convergence rate.
To develop an effective algorithm with simple update and global convergence,
we propose a column-wise update algorithm for solving the SSMF.
This is motivated by the alternating minimization (AM) method, see for instance the book
\cite{B17},  and the proximal alternating linearized minimization (PALM) \cite{BS14} and the block coordinate update \cite{XY13,XY17} for nonconvex and nonsmooth optimization.
We update the $W$-factor column by column via the AM method or the cyclic projected gradient method  and update the $H$-factor column by column by using the projected gradient method, where the involved subproblems can be solved efficiently.
In \cite[Theorem 14.3]{B17}, it has been shown that every limit point of the sequence generated by the AM method (i.e., the block coordinate descent method as in \cite[Section 2.7]{B99}) is a stationary point, which is a coordinate-wise minimum.
A PALM algorithm in \cite{BS14} and the block coordinate update algorithms in \cite{XY13,XY17} were proposed for  solving nonconvex-nonsmooth optimization.
In particular, based on the Kurdyka-{\L}ojasiewicz (KL) property, the proposed PALM algorithm and the block coordinate update algorithms were shown to converge globally to a critical point.
For the sparse NMF,  the $\ell_0$-norm $\|\cdot\|_0$ is a semi-algebraic function, which satisfies the KL property \cite[Theorem 3 and Examples 2--3]{BS14}.
Then we reformulate the SSMF as an unconstrained nonconvex-nonsmooth minimization problem, where the objective function measures the difference between $V$ and $WH$ and the entry-wise nonnegativity and  sparseness of $W$ and $H$. The global convergence of the proposed method is established.  The main advantage of our method lies  in the fact that the sequence generated by our method converges globally to a special critical point, at which the objective function is nearly globally minimized  over each column vector of the $W$-factor and  is globally minimized over the $H$-factor as a whole if the sparseness constraint of $H$ is removed.
Numerical experiments on both synthetic and real data sets show that the proposed algorithm is more effective than the PALM method for solving the SSMF in terms of the reconstruction error.

Throughout this paper, we use the following notations.
Let $\Rmn$ be the set of all $m\times n$ real matrices and $\Rn=\R^{n\times 1}$. Let $\Rn$ be equipped with the Euclidean inner product $\langle \cdot,\cdot\rangle$ and its induced norm $\|\cdot\|$. Let $\Rmn$ be equipped with the Frobenius inner product $\langle \cdot,\cdot\rangle_F$ and its induced Frobenius norm $\|\cdot\|_F$.
The superscript``$\cdot^T$" stands for the  transpose of a matrix or vector. $I_n$ denotes the identity matrix of order $n$.
Let $|\cdot|$ be the absolute value of a real number or the components of a real vector and
$\|\cdot\|_p$ be the matrix $p$-norm (especially $p=1,2,\infty$)  and $\|\cdot \|_0$ denotes the number of nonzero entries of a vector or a matrix. We denote by $\lambda_{\max}(\cdot)$ the largest eigenvalue of a symmetric matrix.
Let $\Pi_\cd(\cdot)$ denote the metric projection onto a subset $\cd$ in $\Rn$ or $\Rmn$. For any matrix $A$, $A\ge 0$ means that $A$ is entrywise nonnegative and $A(i,:)$ and $A(:,j)$ denote respectively the $i$-th row and the $j$-th column of $A$.
For any given point $(X_1,X_2)\in\R^{m_1\times n_1}\times \subset\R^{m_2\times n_2}$ and any subset $\cd_1\times \cd_2\subset\R^{m_1\times n_1}\times \R^{m_2\times n_2}$, the distance from $(X_1,X_2)$ to $\cd_1\times \cd_2$ is determined by
\[
\dist((X_1,X_2),\cd_1\times \cd_2))=\inf\{\|(X_1,X_2)-(Y_1,Y_2)\|_F\;|\; (Y_1,Y_2)\in\cd_1\times \cd_2\}.
\]
Let $[n]=\{1,2,\ldots,n\}$ and for any nonempty set $\cs\subset[n]$, let $\card(\cs)$ and $\overline{\cs}$ be the cardinality of $S$ and the complement of $\cs$ in $[n]$, respectively.
For any set $\cs\subset[n]$, $\bx_\cs$ is the subvector of a vector $\bx$ with components indexed by $\cs$.

The rest of this paper is organized as follows. In Section \ref{sec2}, we reformulate the  SSMF as a nonconvex-nonsmooth minimization problem and then propose  a new column-wise update algorithm for solving it. In Section \ref{sec3}, we establish the global convergence of the proposed algorithm. In Section \ref{sec4}, we report some numerical experiments to illustrate the effectiveness of our method.
Finally, concluding remarks are given in Section \ref{sec5}.

\section{A column-wise update algorithm}\label{sec2}

In this section, we reformulate the SSMF as an unconstrained nonconvex-nonsmooth minimization problem. Then we  propose a column-wise update algorithm for solving the minimization problem.

\subsection{Preliminaries}
We recall the definition of subdifferential (subgradient) for a nonsmooth function in \cite{M06,RW98}.
\begin{definition}\label{app:sg}
Let $h:\Rn\to (-\infty, +\infty]$ be a proper lower semicontinuous (lsc) function. Then, the set
\[
\widehat{\partial} h(\bar{\ba}):=\big\{\br\in\Rn \; |\; \liminf_{\ba\to\bar{\ba}}\frac{h(\ba)-h(\bar{\ba})-\langle \br,\ba-\bar{\ba}\rangle }{\|\ba-\bar{\ba}\|}\geq 0\big\}.
\]
is the presubdifferential or Fr\'{e}chet subdifferential of $h$ at $\bar{\ba}\in \dom~h$ and we set $\widehat{\partial} h(\bar{\ba}):=\emptyset$ if $\bar{\ba}\notin \dom~h$. Moreover, the set
\BE\label{def:lsub}
\partial h(\bar{\ba}):=\big\{\br\in\Rn~|\mbox{~$\exists \ba^k\to\bar{\ba}$, $h(\ba^k)\to h(\bar{\ba})$ and $\br^k\in \widehat{\partial} h(\ba^k)\to\br$ as $k\to \infty$}\big\}
\EE
is the limiting subdifferential of $h$ at $\bar{\ba}\in\Rn$.
\end{definition}

From \cite[Theorem 8.6]{RW98}, it follows that, for any  $\bar{\ba}\in \dom~g$, $\widehat{\partial} h(\bar{\ba})\subset\partial h(\bar{\ba})$ and $\widehat{\partial} h(\bar{\ba})$ is convex and closed while  $\partial h(\bar{\ba})$ is closed.
A point $\bar{\ba}\in \Rn$ is called a {\em critical point} of $h$ if ${\bf 0}\in\partial h(\bar{\ba})$.

\subsection{Reformulation}
As in \cite{LS99}, for the SSMF,  we measure the approximation  quantity by the residual $\|V-WH\|_F$ in the Frobenius norm. Then we can  rewrite the SSMF as the following minimization problem:
\BE\label{pro:smf}
\begin{array}{ll}
\min\limits_{W\in\Rmr,H\in\Rrn} & \displaystyle f(W,H)=\frac{1}{2}\|V-WH\|_F^2\\[2mm]
\mbox{s.t.} & W\in \cg_1\cap \cg_2,\quad H \in \cg_3\cap\cg_4,
\end{array}
\EE
where $\cg_1:=\{W=(w_{ij})\in\Rmr\;|\;  W\ge 0,\; \sum_{i=1}^{m}w_{it}=1,\; t=1, \ldots, r \}$,
$\cg_2:=\{W\in\Rmr\;|\; \|W(:,t)\|_0\le s_1,\; t=1,\ldots,r\}$, $\cg_3:=\{H=(h_{ij})\in\Rrn\;|\;  H\ge 0,\; \sum_{t=1}^{r}h_{tj}=1,\; j=1,\ldots,n\}$, and $\cg_4:=\{H\in\Rrn\;|\; \|H(:,j)\|_0\le s_2,\; j=1,\ldots,n\}$. Here, $s_1,s_2>0$ are prescribed sparsity level.


It is easy to see that the SSMF \eqref{pro:smf} can be reduced to the following unconstrained minimization problem:
\BE\label{ssmf:wh}
\min\limits_{W\in\Rmr,H\in\Rrn} F(W,H):=f(W,H)+\delta_{\cg_1\cap \cg_2}(W)+\delta_{\cg_3\cap\cg_4}(H),
\EE
where $\delta_\ck$ is the indicator function of a set $\ck$, i.e., $\delta_\ck(\bx)=0$ if $\bx\in\ck$ and $\delta_\ck(\bx)=+\infty$ if  $\bx\not\in\ck$.

In Problem \eqref{ssmf:wh}, $f$ is nonconvex and smooth, and both $\delta_{\cg_1\cap\cg_2}$ and $\delta_{\cg_3\cap\cg_4}$ are nonsmooth.
One may solve Problem \eqref{ssmf:wh} by the PALM method in \cite{BS14}, which is described in Algorithm \ref{PALM} below.
\begin{algorithm} 
\caption{A PALM method for Problem \eqref{ssmf:wh}} \label{PALM}
\begin{description}
\item [{Step 0}.] Choose $W^0\in\cg_1\cap\cg_2$, $H^0\in\cg_3\cap\cg_4$, $\delta_1>0$, $\delta_2>0$, and let $k:=0$.
\item[{Step 1}.] Take $\mu_{k}=1/(\|H^{k}\|_F^2+\delta_1)$ and compute
\[
W^{k+1}=\Pi_{\cg_1\cap\cg_2}\big(W^{k}-\mu_{k}\nabla_Wf(W^k, H^k)\big).
\]
\item [{Step 2.}] Take $\nu_{k}=1/(\|W^{k+1}\|_F^2+\delta_2)$ and compute
\[
H^{k+1} = \Pi_{\cg_3\cap\cg_4} \big(H^{k}-\nu_{k} \nabla_H f (W^{k+1}, H^k)\big).
\]
\item [Step 3.] Replace $k$ by $k+1$ and go to Step 1.
\end{description}
\end{algorithm}

In Algorithm \ref{PALM}, both the $W$-factor  and  the $H$-factor are updated as a whole, which may cause slow convergence. To overcome this drawback, we develop an efficient algorithm with simpler update for solving the SSMF \eqref{pro:smf}.

In the following, we present a column-wise update algorithm for solving the SSMF \eqref{pro:smf}, where both the $W$-factor  and  the $H$-factor are updated column-by-column.
To begin, we let
\[
\cp_n:=\{\bx\in \Rn \; |\; \bx\ge 0,\;\sum_{j=1}^{n} x_j=1\},
\quad \cq_n^s:=\{\bx\in \Rn \; |\; \|\bx\|_0\leq s\}.
\]
For convenience, we let
\BE\label{def:f}
f(W,H)=\frac{1}{2}\big\|V-
[\bw_1,\ldots,\bw_r]
[\bh_1,\ldots,\bh_n]
\big\|_F^2\equiv f(\bw_1,\ldots,\bw_r, \bh_1,\ldots,\bh_n),
\EE
for all $W:=[\bw_1,\ldots,\bw_r]\in\Rmr$ and $H:=[\bh_1,\ldots,\bh_n]\in\Rrn$. Also, we define $F$ by
\begin{eqnarray}\label{def:Ffd}
F(W,H) &=&f(\bw_1,\ldots,\bw_r, \bh_1,\ldots,\bh_n) +\sum_{i=1}^r\delta_{\cp_m}(\bw_i)+\sum\limits_{i=1}^r\delta_{\cq_m^{s_1}}(\bw_i)\nonumber\\
&&+\sum\limits_{t=1}^n\delta_{\cp_r}(\bh_t)+\sum\limits_{t=1}^n\delta_{\cq_r^{s_2}}(\bh_t)
\equiv F(\bw_1,\ldots,\bw_r, \bh_1,\ldots,\bh_n),
\end{eqnarray}
for all $W:=[\bw_1,\ldots,\bw_r]\in\Rmr$ and $H:=[\bh_1,\ldots,\bh_n]\in\Rrn$.
Then the SSMF \eqref{pro:smf} can be rewritten  equivalently as the following unconstrained minimization problem:
\BE\label{pro:smf-eq}
\begin{array}{cc}
\min\limits_{\bw_1,\ldots,\bw_r\in\R^{m}, \bh_1,\ldots,\bh_n\in\R^{r}} & \displaystyle F(\bw_1,\ldots,\bw_r, \bh_1,\ldots,\bh_n).
\end{array}
\EE
To avoid confusion, we refer to Problem \eqref{pro:smf-eq} as the SSMF.

Let $V:=[\bv_1,\ldots,\bv_n]$.
Then the partial gradient $\nabla_{\bw_i} f$ of $f$ defined in
\eqref{def:f} at $(\bw_1,\ldots,\bw_r$, $\bh_1,\ldots,\bh_n)\in\underbrace{\Rm\times \cdots\times\Rm}_{r}\times \underbrace{\Rr\times \cdots\times\Rr}_{n}$ is given by
\BE\label{f:grad}
\nabla_{\bw_i} f(\bw_1,\ldots,\bw_r,\bh_1,\ldots,\bh_n)=-(U_i-\bw_iH(i,:))(H(i,:))^T
\EE
with the Lipschitz constant
$\|H(i,:)\|^2$ for $i=1,\ldots,r$,
where
$U_i:=V-\sum_{j=1}^{i-1}\bw_jH(j,:)-\sum_{j=i+1}^{r}\bw_jH(j,:)$ for $1\le i \le r$.
Moreover,  the partial gradient $\nabla_{\bh_t} f$ of $f$ at $(\bw_1,\ldots,\bw_r,\bh_1,\ldots,$ $\bh_n)\in\underbrace{\Rm\times \cdots\times\Rm}_{r}\times \underbrace{\Rr\times \cdots\times\Rr}_{n}$ is given by
\BE\label{f:gradh}
\nabla_{\bh_t} f(\bw_1,\ldots,\bw_r,\bh_1,\ldots,\bh_n)=W^T(W\bh_t-\bv_t),
\EE
with  the Lipschitz constant $\|W^TW\|_2$  for $t=1,\ldots,n$.

We must point out  that the function $F$ defined in  \eqref{def:Ffd} is nonconvex-nonsmooth as $f$ is nonconvex and smooth while $\delta_{\cp_m}$, $\delta_{\cp_r}$, $\delta_{\cq_m^{s_1}}$ and $\delta_{\cq_r^{s_2}}$ are all nonsmooth. For the partial subdifferential of $F$ defined by \eqref{def:Ffd},
we have the following result from \cite{AB10,BS14,RW98}.
\begin{lemma}\label{app:ps}
Let $F$ be defined in  \eqref{def:Ffd}. Then for any
\[
(W,H):=(\bw_1,\ldots,\bw_r,\bh_1,\ldots,\bh_n) \in\underbrace{\Rm\times \cdots\times\Rm}_{r}\times \underbrace{\Rr\times \cdots\times\Rr}_{n},
\]
we have $\partial F(W,H)=(\partial_{\bw_1}F(W,H),\ldots,\partial_{\bw_r}F(W,H),
\partial_{\bh_1}F(W,H),\ldots,\partial_{\bh_n}F(W,H))$, where
$\partial_{\bw_i}F$ $(\bw_1,\ldots,\bw_r,\bh_1,\ldots,\bh_n)=\{\nabla_{\bw_i} f(\bw_1,\ldots,\bw_r,\bh_1,\ldots,\bh_n)+\partial\delta_{\cp_m}(\bw_i)+\partial\delta_{\cq_m^{s_1}}(\bw_i)\}
$
for $i=1,\ldots,r$ and
$
\partial_{\bh_t} F(\bw_1,\ldots,\bw_r,\bh_1,\ldots,\bh_n)=\{\nabla_{\bh_t} f(\bw_1,\ldots,\bw_r,$ $\bh_1,\ldots,\bh_n) +\partial\delta_{\cp_r}(\bh_t) +\partial\delta_{\cq_r^{s_2}}(\bh_t)\}
$
for $t=1,\ldots,n$.
\end{lemma}

We note that  $\delta_{\cp_m}$, $\delta_{\cp_r}$, $\delta_{\cq_m^{s_1}}$ and $\delta_{\cq_r^{s_2}}$ are all  lower semicontinuous (lsc).
As noted in  \cite{BS14},  the $\ell_0$-norm $\|\cdot\|_0$ is
semi-algebraic and thus satisfies the KL property (see for instance \cite{AB10,BS14,XY13,XY17}). Hence, $F$ is semi-algebraic  and thus is a KL function.
It is natural to solve the SSMF \eqref{pro:smf-eq} by the PALM method in \cite{BS14} and the generated sequence converges to a critical point of $F$.  However,  for any fixed $W=[\bw_1,\ldots,\bw_r]\in\Rmr$,  $F$ is separable with respect to the variables $\bh_t$'s, i.e.,
      \begin{eqnarray*}
      F(W,H)&=&\sum_{t=1}^n\frac{1}{2}\|\bv_t-W\bh_t\|^2
      +\sum_{i=1}^r\delta_{\cp_m}(\bw_i)+\sum\limits_{i=1}^r\delta_{\cq_m^{s_1}}(\bw_i) \\
     &&  +\sum\limits_{t=1}^n\delta_{\cp_r}(\bh_t) +\sum\limits_{t=1}^n\delta_{\cq_r^{s_2}}(\bh_t).
      \end{eqnarray*}
Then it is desired to n update $\bh_t$'s separately via the proximal gradient method. On the other hand, for any fixed $H=[\bh_1,\ldots,\bh_n]\in\Rrn$, $F$ is non-separable with respect to the variables $\bw_i$'s.
Thus, it is preferred to apply the AM method to updating $\bw_i$'s. However, one can see from \cite[Chapter 14]{B17} and \cite[Section 2.7]{B99}) that, when we directly apply the AM method to the SSMF \eqref{pro:smf-eq}, it is not guaranteed that every limit point of the sequence generated by the AM method is a column-wise minimum point, which is  not necessarily  a stationary point since  $F$ is nonconvex-nonsmooth. In this paper, to take full advantage of both the AM method and the PALM method, we propose a column-wise update algorithm for solving the SSMF \eqref{pro:smf-eq}, which is stated in Algorithm \ref{alg1} below.

\begin{algorithm}
\caption{A column-wise update algorithm for solving  the SSMF \eqref{pro:smf-eq}} \label{alg1}
\begin{description}
\item [{Step 0}.] Choose $W^0:=[\bw_1^0,\ldots,\bw_r^0]\in\cg_1\cap\cg_2$, $H^0:=[\bh_1^0,\ldots,\bh_n^0] \in\cg_3\cap\cg_4$, $\delta_1>0$, $\delta_2>0$, $c\ge 1/\delta_2$, and let $k:=0$.
\item [{Step 1.}] For $i=1,\ldots,r$, if $\|H^{k}(i,:)\|=0$, then set $\bw_i^{k+1}=\bw_i^k$; else compute
\BE\label{eqw1}
\bar{\bw}_i^{k}=\argmin_{\bw_i\in\cp_m\cap\cq_m^{s_1}} \phi_{k}(\bw_i)
\EE
such that if $\phi_{k}(\bw_i^k) -\phi_{k}(\bar{\bw}_i^k)\geq \frac{\delta_1}{2}\|\bw_i^k-\bar{\bw}_i^{k}\|^2$, then set $\bw_i^{k+1}=\bar{\bw}_i^{k}$; else set
\BE\label{eqw2}
\bw_i^{k+1} = \Pi_{\cp_m\cap\cq_m^{s_1}} \big(\bw_i^{k}-\mu_{ik} \nabla\phi_{k}(\bw_i^k)\big),
\EE
where $\phi_{k}(\bw_i):=f(\bw_1^{k+1},\ldots,\bw_{i-1}^{k+1},\bw_i,\bw_{i+1}^k,\ldots,\bw_r^k, H^k)$,
$\Phi_{k}(\bw_i):=\phi_{k}(\bw_i) +\delta_{\cp_m}(\bw_i)+\delta_{\cq_m^{s_1}}(\bw_i)$,
and
$\mu_{ik}=1/(\|H^{k}(i,:)\|^2+\delta_1).$
Set
\[
W^{k+1}:=[\bw_1^{k+1},\ldots,\bw_r^{k+1}].
\]
\item[{Step 2}.] For $t=1,\ldots,n$, take $\nu_{tk}=\min\big\{c,\frac{\|\nabla\psi_k(\bh_t^k)\|^2} {\|W^{k+1}\nabla\psi_k(\bh_t^k)\|^2}\big\}$ and compute
\BE\label{eqh1}
\bar{\bh}_t^{k}=\Pi_{\cp_r\cap\cq_r^{s_2}}\big(\bh_t^{k}-\nu_{tk}\nabla\psi_k(\bh_t^k)\big)
\EE
such that if $\psi_{k}(\bh_t^{k})-\psi_{k}(\bar{\bh}_t^{k})\ge \frac{\delta_2}{2}\|\bh_t^{k}-\bar{\bh}_t^{k}\|^2$, then set $\bh_t^{k+1}=\bar{\bh}_t^{k}$; else set
\BE\label{eqh2}
\bh_t^{k+1}=\Pi_{\cp_r\cap\cq_r^{s_2}}\big(\bh_t^{k}-\nu_{tk}\nabla\psi_k(\bh_t^k)\big),
\EE
where $\psi_{k}(\bh_t):=f(W^{k+1}, \bh_1^k,\ldots,\bh_{t-1}^k,\bh_t,\bh_{t+1}^k,\ldots,\bh_n^k,),$
$\Psi_{k}(\bh_t):=\psi_{k}(\bh_t)+\delta_{\cp_r}(\bh_t)+\delta_{\cq_r^{s_2}}(\bh_t)$, and $\nu_{tk}=1/(\|(W^{k+1})^TW^{k+1}\|_2+\delta_2).$
Set
\BE\label{H}
H^{k+1}:=[\bh_1^{k+1},\ldots,\bh_n^{k+1}].
\EE

\item [Step 3.] Replace $k$ by $k+1$ and go to Step 1.
\end{description}
\end{algorithm}

Regarding Algorithm \ref{alg1}, we have the following remarks.
\begin{itemize}
  \item In Step 2 of  Algorithm \ref{alg1}, we have by the definition of $\psi_k$, for any $1\le t\le n$,
      \[
      \psi_{k}(\bh_t^{k})-\psi_{k}(\bar{\bh}_t^{k})=\frac{1}{2}\big(\|\bv_t-W^{k+1}\bh_t^k\|^2
      -\|\bv_t-W^{k+1}\bar{\bh}_t^k\|^2\big),
      \]
      which is easy to estimate.

\item In Step 1 of  Algorithm \ref{alg1}, for each $1\le i\le r$, we update $\bw_i^k$ by \eqref{eqw1}  if the minimum in \eqref{eqw1} is uniquely attained, which satisfies the sufficient decrease condition
$\phi_{k}(\bw_i^k) -\phi_{k}(\bar{\bw}_i^k)\geq \frac{\delta_1}{2}\|\bw_i^k-\bar{\bw}_i^{k}\|^2;$
Otherwise, we update $\bw_i^k$ via the proximal gradient scheme \eqref{eqw2}.

  \item For the scalar $\nu_{tk}$ defined in Step 2 of  Algorithm \ref{alg1}, we have the following bounds.
Since $W^{k+1}=[\bw_1^{k+1},\ldots,\bw_r^{k+1}]$ with $\bw_i^{k+1}\in\cp_m\cap\cq_m^{s_1}$ for $i=1,\ldots,r$, we have $\|W^{k+1}\|_1=1$ and $\|W^{k+1}\|_\infty=\max_{1\le i\le m}\sum_{j=1}^rw_{ij}\le r$. Thus,
\begin{small}
      \BE\label{h:2n}
      \|(W^{k+1})^TW^{k+1}\|_2\le \|W^{k+1}\|_2^2\le \|W^{k+1}\|_1\|W^{k+1}\|_\infty=\|W^{k+1}\|_\infty\le r,
      \EE
\end{small}
      If $\nu_{tk}$ is determined by \eqref{eqh1}, then from  \eqref{h:2n} we have
      \BE\label{mu:bd1}
     \nu_{tk}= c\ge 1/\delta_2\quad\mbox{or}\quad 1/r\le 1/\|(W^{k+1})^TW^{k+1}\|_2\le \nu_{tk}\le c.
       \EE
       If $\nu_{tk}$ is determined by \eqref{eqh2}, then from  \eqref{h:2n} we have
       \BE\label{mu:bd2}
       1/(r+\delta_2)\le 1/( \|(W^{k+1})^TW^{k+1}\|_2+\delta_2)= \nu_{tk}\le 1/\delta_2.
      \EE
  \item All subproblems in  Algorithm \ref{alg1} can be solved efficiently (see Section \ref{sec31}).
\end{itemize}

\section{Convergence analysis}\label{sec3}
In this section, we  first derive the projection onto $\cp_n\cap\cq_n^s$. Then we establish  the global convergence of Algorithm \ref{alg1}  by combining the convergence analysis of the PALM \cite{BS14} and that of the AM method (e.g., \cite[Chapter 14]{B17} and \cite[Section 2.7]{B99}) based on the KL property. Here, we give a closed-form  projection onto $\cp_n\cap\cq_n^s$ (see Proposition \ref{pro:proj}), which can be computed with the computational complexity $O(n\log n)$. We also give a  closed-form of $\bar{\bw}_i^{k}$ defined in \eqref{eqw1} (Lemma \ref{lem:wkbar}).
\subsection{Projection onto $\cp_n\cap\cq_n^s$}\label{sec31}
We note that, for any given $\by\in\Rn$,  $\bz=\Pi_{\cp_n\cap\cq_n^s}(\by)$ is the unique solution to the following minimization problem:
\BE\label{projs}
\begin{array}{ll}
 \min\limits_{\substack{\cs\subset [n],\; \card(\cs)=s}}  &  \displaystyle  \frac{1}{2}\|\bz-\by\|^2  \\
 \mbox{s.t.} &\bz\in\cm_\cs^n:=\{\bz\in\Rn\;| \; \sum_{i\in \cs}z_i=1,\; \bz_\cs\ge 0,\;  \bz_{\overline{\cs}}= \bf 0\}.
 \end{array}
\EE

For problem (\ref{projs}), we have  the following result.
\begin{proposition}\label{pro:proj}
Let $\by\in\Rn$ and $1\le s\le n$. Suppose  $\pi=\{\pi(1),\ldots,\pi(n)\}$ is a permutation such that $y_{\pi(1)}\ge y_{\pi(2)}\ge \cdots\ge y_{\pi(n)}$.  Let $\cs_*=\{\pi(1),\ldots,\pi(s)\}\subset [n]$. Then a solution $\bz^*\in\Rn$ to problem {\rm  (\ref{projs})} is given by
\BE\label{proj:sstar}
\bz^*:= \argmin_{\bz\in\cm_{\cs_*}^n}   \frac{1}{2}\|\bz-\by\|^2.
\EE
\end{proposition}

\begin{proof}
Without loss of generality,
we assume that $y_1\ge y_2\ge\cdots\ge y_n$.
In this case, $\cs_*=\{1,\ldots,s\}$.
For any $\cs=\{j_1,\ldots,j_s\}\subset [n]$ with $j_1<j_2<\cdots<j_s$, there exists a permutation matrix $P^{(\cs)}\in\Rnn$ such that $\bz\in\cm_\cs^n$ if and only if $\bz=P^{(\cs)}\bw$ for some $\bw\in\cm_{\cs_*}^n$. Let
$\hat{\bz}:= \argmin_{\bz\in\cm_\cs^n}   \frac{1}{2}\|\bz-\by\|^2$. Then we have $\hat{\bz}=P^{(\cs)}\hat{\bw}\in\cm_\cs^n$, where $\hat{\bw}:= \argmin_{\bw\in\cm_{\cs_*}^n}   \frac{1}{2}\|\bw-(P^{(\cs)})^T\by\|^2$.
By hypothesis, $y_1\ge y_2\ge\cdots\ge y_n$ and $j_1<j_2<\cdots<j_s$. Thus, $y_i\ge y_{j_i}$ for $i=1,\ldots,s$. We note that  $\hat{z}_{j_i}\ge 0$ for $i=1,\ldots,s$ since
$\hat{\bz}\in\cm_\cs^n$. Hence, $-\hat{z}_{j_i}y_i \le -\hat{z}_{j_i}y_{j_i}$ for $i=1,\ldots,s$.
By the definitions of $\bz^*$ and $\hat{\bw}$ and using the orthogonal invariance of the vector $2$-norm, we have
\begin{eqnarray*}\label{eq:upb}
\frac{1}{2}\|\bz^*-\by\|^2 &\le& \frac{1}{2}\|\hat{\bw}-\by\|^2
=\frac{1}{2}\|\hat{\bz}-P^{(\cs)}\by\|^2 =\frac{1}{2}\sum_{i=1}^{s}(\hat{z}_{j_i}-y_i)^2+\frac{1}{2}\sum_{i=s+1}^{n}y_i^2 \\
&=& \frac{1}{2}\sum_{i=1}^{s}(\hat{z}_{j_i}^2-2\hat{z}_{j_i}y_i)+\frac{1}{2}\sum_{i=1}^{n}y_i^2
\le \frac{1}{2}\sum_{i=1}^{s}(\hat{z}_{j_i}^2-2\hat{z}_{j_i}y_{j_i})+\frac{1}{2}\sum_{i=1}^{n}y_i^2\\
&=&  \frac{1}{2}\sum_{i=1}^{s}(\hat{z}_{j_i}-y_{j_i})^2+\frac{1}{2}\sum_{i\notin\cs}y_{j_i}^2
= \frac{1}{2}\|\hat{\bz}-\by\|^2.
\end{eqnarray*}
By the arbitrariness of $\cs$,  the proof is complete. 
\end{proof}

\begin{remark}
In \cite[Theorem 2.4]{XL16}, Xu et al. gave another way for finding a closed-form projection onto $\cp_n\cap\cq_n^s$.
 \end{remark}

Based on Proposition \ref{pro:proj} and the projection of a vector onto the probability simplex \cite{CY11}, we can find $\bz=\Pi_{\cp_n\cap\cq_n^s}(\by)$ of a given vector $\by\in\Rn$, which  is described in Algorithm \ref{alg:proj}.
\begin{algorithm}
\caption{Projection onto $\cp_n\cap\cq_n^s$}\label{alg:proj}
\begin{description}
\item [{Step 0.}] Given a vector $\by=(y_1,\ldots,y_n)^T\in\Rn$ and a permutation $\pi=\{\pi(1),\ldots,\pi(n)\}$ such that $y_{\pi(1)}\ge y_{\pi(2)}\ge \cdots\ge y_{\pi(n)}$.
\item [{Step 1.}]
Find $\rho=\max_j \{j\in[s]\; | \; y_{\pi(j)}-\frac{1}{j} \big(\sum_{r=1}^j y_{\pi(r)}-1\big)>0 \}$ and define $\beta=\frac{1}{\rho}\big(\sum_{r=1}^{\rho} y_{\pi(r)}-1\big)$.
\item [{Step 2.}] Set $\bz:=(z_1,\ldots,z_n)^T\in\Rn$ with $z_{\pi(j)}=\max\{y_{\pi(j)}-\beta,0\}$ for $j=1,\ldots,s$ and $z_{\pi(j)}=0$ for $j=s+1,\ldots,n$.
\end{description}
\end{algorithm}

We point out that the complexity of Algorithm \ref{alg:proj} is $O(n\log n)$, which is dominated by sorting the elements of $\by$.
\subsection{Global convergence of Algorithm \ref{alg1}} \label{sec33}

In this subsection, we establish the global convergence of  Algorithm \ref{alg1}. In the following, we  present some necessary lemmas.
We first recall the lemma for a continuously differentiable function (\cite{B99,OR70}).
\begin{lemma}\label{lem:dl}
Let $g:\Rn\to \R$ be a continuously differentiable function, where the gradient $\nabla g$ is Lipschitz-continuous with  a Lipschitz constant $L_g$. Then one has
\[
g(\by_2)\le g(\by_1)+\langle\by_2-\by_1,\nabla g(\by_1)\rangle + \frac{L_g}{2}\|\by_2-\by_1\|^2,\quad\forall \by_1,\by_2\in\Rn.
\]
\end{lemma}

For the sequence $\{\bar{\bw}_i^{k}\}$ generated by  Algorithm \ref{alg1}, we have the following result.
\begin{lemma}\label{lem:wkbar}
Let $(W^k,H^k)$ be the current iterate generated by  Algorithm {\rm \ref{alg1}}. For $i=1,\ldots,r$, if $\|H^{k}(i,:)\|>0$, then
\begin{small}
\[
\bar{\bw}_i^{k}=\Pi_{\cp_m\cap\cq_m^{s_1}}\Big(\frac{1}{\|H^{k}(i,:)\|^2}U_i^k(H^{k}(i,:))^T\Big)
=\Pi_{\cp_m\cap\cq_m^{s_1}}\Big(\bw_i^k-\frac{1}{\|H^{k}(i,:)\|^2}\nabla_{\bw_i}\phi_k(\bw_i^k)\Big),
\]
\end{small}
where $U_i^k:=V-\sum_{j=1}^{i-1}\bw_j^{k+1}H^k(j,:)-\sum_{j=i+1}^{r}\bw_j^kH^k(j,:).$
\end{lemma}

\begin{proof}
By the definition of $\bar{\bw}_i^{k}$, we have
\begin{eqnarray*}
\bar{\bw}_i^{k}
=\argmin_{\bw_i\in\cp_m\cap\cq_m^{s_1}} \frac{1}{2}\|U_i^k-\bw_iH^{k}(i,:)\|_F^2.
\end{eqnarray*}
Let $1\le i\le r$ be fixed.
By hypothesis, we have $\|H^{k}(i,:)\|>0$.
Then one can construct an orthogonal matrix
\[
Q=\left[\frac{(H^{k}(i,:))^T}{\|H^{k}(i,:)\|},\bq_2,\ldots,\bq_m\right]\in\Rmm.
\]
By \eqref{f:grad} we have
\[
\nabla_{\bw_i}\phi_k(\bw_i^k)=-\big(U_i^k-\bw_i^kH^k(i,:) \big)\big(H^{k}(i,:)\big)^T=
-U_i^k(H^{k}(i,:))^T+ \|H^{k}(i,:)\|^2\bw_i^k.
\]
This, together with the orthogonal invariance of the Frobenius matrix norm, yields
\begin{eqnarray*}
\bar{\bw}_i^{k}&=&\argmin_{\bw_i\in\cp_m\cap\cq_m^{s_1}} \frac{1}{2}\Big\|U_i^kQ-\bw_iH^{k}(i,:)Q\Big\|_F^2\\
&=& \argmin_{\bw_i\in\cp_m\cap\cq_m^{s_1}}\frac{1}{2}\|H^{k}(i,:)\|^2 \Big\|\bw_i- \frac{1}{\|H^{k}(i,:)\|^2}U_i^k(H^{k}(i,:))^T\Big\|^2 \\
&=& \argmin_{\bw_i\in\cp_m\cap\cq_m^{s_1}}\frac{\|H^{k}(i,:)\|^2}{2} \Big\|\bw_i-\Big(\bw_i^k-\frac{1}{\|H^{k}(i,:)\|^2}\nabla_{\bw_i}\phi_k(\bw_i^k)\Big)\Big\|^2.
\end{eqnarray*}
\end{proof}

Lemma \ref{lem:wkbar} shows the minimization in \eqref{eqw1} is in fact a projection onto $\cp_m\cap\cq_m^{s_1}$, which can be computed via Algorithm \ref{alg:proj}.

We also have the following useful properties for Algorithm {\rm \ref{alg1}}.

\begin{lemma}\label{lem:fwhkn}
Let $\{(W^k,H^k)\}$ be the sequence generated by  Algorithm {\rm \ref{alg1}}. Then for any $k\ge 0$,
\BE\label{ineq:fwkn}
\Phi_{k}(\bw_i^{k})-\Phi_{k}(\bw_i^{k+1})\ge \frac{\delta_1}{2}\|\bw_i^{k}-\bw_i^{k+1}\|^2,
\EE
for $i=1,2,\ldots, r$
and
\begin{eqnarray}\label{ineq:fhkn}
\Psi_{k}(\bh_t^k) -\Psi_{k}(\bh_t^{k+1})\geq \frac{\delta_2}{2}\|\bh_t^k-\bh_t^{k+1}\|^2,
\end{eqnarray}
for $t=1,2,\ldots, n$.
\end{lemma}

\begin{proof}
We first prove that \eqref{ineq:fwkn} holds for all $k\ge 0$.
Let $k\ge 0$ and $1\le i\le r$ be fixed. We establish  \eqref{ineq:fwkn}  in different cases.
If $\|H^{k}(i,:)\|=0$, then $\bw_i^{k+1}=\bw_i^{k}$.
In this case,  it is obvious that  \eqref{ineq:fwkn} holds.

Suppose $\|H^{k}(i,:)\|\neq 0$.
If $\Phi_{k}(\bw_i^k) -\Phi_{k}(\bar{\bw}_i^k)\geq \frac{\delta_1}{2}\|\bw_i^k-\bar{\bw_i}^{k}\|^2$,
then
$\bw_i^{k+1}=\bar{\bw}_i^k$ and thus  \eqref{ineq:fwkn} also holds.
Otherwise, $\bw_i^{k+1}$ is determined by \eqref{eqw2}, i.e.,
\[
\bw_i^{k+1}= \argmin_{\bw_i\in \Rm}\big\{\big\langle \bw_i-\bw_i^k, \nabla\phi_k(\bw_i^k)\big\rangle+\frac{1}{2\mu_{ik}}\|\bw_i-\bw_i^k\|^2
+\delta_{\cp_m}(\bw_i)+\delta_{\cq_m^{s_1}}(\bw_i)\big\}.
\]
Thus,
\begin{eqnarray}\label{dl:fhlkn}
 \delta_{\cp_m}(\bw_i^k)+\delta_{\cq_m^{s_1}}(\bw_i^k) &\ge& \big\langle \bw_i^{k+1}-\bw_i^k, \nabla\phi_k(\bw_i^k)\big\rangle + \frac{1}{2\mu_{ik}}\|\bw_i^{k+1}-\bw_i^k\|^2 \nonumber\\
 && +\delta_{\cp_m}(\bw_i^{k+1}) +\delta_{\cq_m^{s_1}}(\bw_i^{k+1}).\nonumber
\end{eqnarray}
This, together with Lemma \ref{lem:dl} for $g=\phi_k$, yields
\begin{eqnarray*}
&&\phi_{k}(\bw_i^{k+1})+\delta_{\cp_m}(\bw_i^{k+1})+\delta_{\cq_m^{s_1}}(\bw_i^{k+1}) \\
&\le& \phi_{k}(\bw_i^{k}) +\delta_{\cp_m}(\bh_i^k)+\delta_{\cq_m^{s_1}}(\bw_i^k) -\frac{1}{2}\Big(\frac{1}{\mu_{ik}}-\|H^{k}(i,:)\|^2)\Big) \|\bw_i^{k+1}-\bw_i^k\|^2.
\end{eqnarray*}
By hypothesis, $\mu_{ik}=1/(\|H^{k}(i,:)\|^2+\delta_1)$.
Then we obtain  \eqref{ineq:fwkn}.

Next, we show that  \eqref{ineq:fhkn} holds for $t=1,2,\ldots, n$ and for all $k\ge 0$.
Let $1\le l \le r$ and $k\ge 0$ be fixed.
If $\Psi_{k}(\bh_t^{k})-\Psi_{k}(\bar{\bh}_t^{k})\ge \frac{\delta_2}{2}\|\bh_t^{k}-\bar{\bh}_t^{k}\|^2$, then we have $\bh_t^{k+1}=\bar{\bh}_t^{k}$.
In this case, it is natural that \eqref{ineq:fhkn} holds.
Otherwise, $\bh_t^{k+1}$ is given by \eqref{eqh2}, i.e.,
\[
\bh_t^{k+1}=\argmin_{\bh_t\in \Rr}\big\{\big\langle \bh_t-\bh_t^k, \nabla\psi_k(\bh_t^k)\big\rangle +\frac{1}{2\nu_{tk}}\|\bh_t-\bh_t^k\|^2+\delta_{\cp_r}(\bh_t)+\delta_{\cq_r^{s_2}}(\bh_t)\big\},
\]
which gives rise to
\begin{eqnarray}\label{dl:fwhkn}
\delta_{\cp_r}(\bh_t^k)+\delta_{\cq_r^{s_2}}(\bh_t^k) &\ge&
\big\langle \bh_t^{k+1}-\bh_t^k, \nabla\psi_k(\bh_t^k)\big\rangle +\frac{1}{2\nu_{tk}}\|\bh_t^{k+1}-\bh_t^k\|^2  \nonumber\\
&& +\delta_{\cp_r}(\bh_t^{k+1})+\delta_{\cq_r^{s_2}}(\bh_t^{k+1}).\nonumber
\end{eqnarray}
This, together with  Lemma \ref{lem:dl} for $g=\psi_{k}$, yields
\begin{eqnarray*}
&&\psi_{k}(\bh_t^{k+1})+\delta_{\cp_r}(\bh_t^{k+1})+\delta_{\cq_r^{s_2}}(\bh_t^{k+1}) \\
&\le& \psi_k(\bh_t^{k}) +\delta_{\cp_r}(\bh_t^k)+\delta_{\cq_r^{s_2}}(\bh_t^k) -\frac{1}{2}\Big(\frac{1}{\nu_{tk}}-\|(W^{k+1})^TW^{k+1}\|_2\Big)\|\bh_t^{k+1}-\bh_t^k\|^2.
\end{eqnarray*}
By hypothesis, $\nu_{tk}=1/(\|(W^{k+1})^TW^{k+1}\|_2+\delta_2)$. Therefore, we have \eqref{ineq:fhkn}.
\end{proof}

The following lemma shows the monotone decreasing property of the sequence $\{F(W^{k},H^{k})\}$ generated by Algorithm \ref{alg1}.

\begin{lemma}\label{lem:fkn}
The sequence $\{F(W^{k},H^{k})\}$ generated by Algorithm \ref{alg1}  is monotonic decreasing and for all $k\ge 0$,
\[
F(W^k,H^k)-F(W^{k+1},H^{k+1})\ge\frac{\delta_1}{2}\|W^{k+1}-W^k\|_F^2+\frac{\delta_2}{2}\|H^{k+1}-H^k\|_F^2.
\]
\end{lemma}

\begin{proof}  Using  Lemma \ref{lem:fwhkn} and the definitions of $\Phi_{k}$ and $\Psi_{k}$ we have, for any $k\ge 0$,
\begin{small}
\begin{eqnarray*}\label{ineq:fwhkn}
&&F(W^k,H^k)-F(W^{k+1},H^{k+1}) \nonumber \\
&=&\sum\limits_{i=1}^r \big(\Phi_{k}(\bw_i^{k})-\Phi_{k}(\bw_i^{k+1})\big)
+\sum\limits_{t=1}^n\big(\Psi_{k}(\bh_t^{k})-\Psi_{k}(\bh_t^{k+1})\big) \nonumber \\
&\ge&\frac{\delta_1}{2}\sum\limits_{i=1}^r\|\bw_i^{k+1}-\bw_i^k\|^2+\frac{\delta_2}{2}\sum\limits_{t=1}^n\|\bh_t^{k+1}-\bh_t^k\|^2
=\frac{\delta_1}{2}\|W^{k+1}-W^k\|_F^2+\frac{\delta_2}{2}\|H^{k+1}-H^k\|_F^2.\nonumber
\end{eqnarray*}
\end{small}
This shows that  $\{F(W^{k},H^{k})\}$ is monotonic decreasing and bounded below.
\end{proof}
\begin{remark}\label{rem:fkn}
From Lemma \ref{lem:fkn}, it follows that, for any integer $q>0$,
\begin{eqnarray*}
&&\sum_{k=0}^q\big\|(W^{k+1},H^{k+1})-(W^{k},H^{k})\big\|_F^2
\le \frac{2}{\delta_{3}} \big(F(W^0,H^0)-F(W^{q+1},H^{q+1}) \big),
\end{eqnarray*}
where $\delta_{3}:=\min\{\delta_1,\delta_2\}$. Hence, $\sum_{k=0}^\infty\big\|(W^{k+1},H^{k+1})-(W^{k},H^{k})\big\|_F^2<\infty$.
\end{remark}

The following lemma gives a subgradient lower bound for the successive iterate gap. We refer to Definition \ref{app:sg} and Lemma \ref{app:ps} for the subgradient of $F$.
\begin{lemma}\label{lem:abk}
Let $\{W^k,H^k\}$ be the sequence generated by Algorithm \ref{alg1}. For any $k\ge 0$, define
$A^{k+1}=[\ba_1^{k+1}, \ldots, \ba_r^{k+1}]$ with
\[
\ba_i^{k+1}:=\nabla_{\bw_i} f(W^{k+1},H^{k+1})-\nabla\phi_k(\bw_i^{k})-\frac{1}{\alpha_{ik}}(\bw_i^{k+1}-\bw_i^k),
\]
for $i=1, \ldots, r$ and $B^{k+1}=[\bb_1^{k+1}, \ldots, \bb_r^{k+1}]$ with
\[
\bb_t^{k+1}:=\nabla_{\bh_t}f(W^{k+1}, H^{k+1})-\nabla\psi_k(\bh_t^{k}) -\frac{1}{\nu_{tk}}(\bh_t^{k+1}-\bh_t^k),
\]
for $t=1, \ldots, n$, where $\alpha_{ik}=1$ if $\|H^{k}(i,:)\|=0$; else
$\alpha_{ik}=1/\|H^{k}(i,:)\|^2$ if
$\Phi_{k}(\bw_i^k) -\Phi_{k}(\bar{\bw}_i^k)\geq \frac{\delta_1}{2}\|\bw_i^k-\bar{\bw}_i^{k}\|^2$
and
$\alpha_{ik}=\mu_{ik}$ if
$\Phi_{k}(\bw_i^k) -\Phi_{k}(\bar{\bw}_i^k)< \frac{\delta_1}{2}\|\bw_i^k-\bar{\bw}_i^{k}\|^2.$
Then $(A^{k+1}, B^{k+1}) \in \partial F(W^{k+1}, H^{k+1})$ and
\BE\label{bd:abk}
\|(A^{k+1},B^{k+1})\|_F\le \sqrt{\delta_4} \|(W^{k+1},H^{k+1})-(W^k,H^k)\|_F,
\EE
where $\delta_4:=\max\{2(2\sqrt{nr}+\|V\|_2)^2+(\delta_2+2r)^2,r(r+1)(2n+\delta_1)^2\}$.
\end{lemma}
\begin{proof}
Let $k\ge 0$  be fixed.
First, we show that $A^{k+1}\in\partial_W F(W^{k+1}, H^{k+1})$ by proving that $\ba_i^{k+1}\in \partial_{\bw_i} F(W^{k+1}, H^{k+1})$ for $i=1, \ldots, r$. For any fixed $1\le i\le r$, we show that $\ba_i^{k+1}\in \partial_{\bw_i} F(W^{k+1}, H^{k+1})$ in different cases. If  $\|H^{k}(i,:)\|=0$, then we have
\[
\nabla\phi_k(\bw_i^k)=-\big(U_i^k-\bw_i^k H^{k}(i,:)\big)(H^{k}(i,:))^T=\bf 0.
\]
where $U_i^k$ is defined as in Lemma \ref{lem:wkbar}.  By Algorithm \ref{alg1}, we have $\bw_i^{k+1}=\bw_i^{k}$. Since ${\bf 0}$ $\in \partial\delta_{\cp_m}(\bw_i^{k+1})$ and ${\bf 0}\in \partial\delta_{\cq_m^{s_1}}(\bw_i^{k+1})$, we know that $\nabla_{\bw_i} f(W^{k+1},H^{k+1})\in  \partial_{\bw_i} F(W^{k+1}, H^{k+1})$. Thus,
$\ba_i^{k+1}= \nabla_{\bw_i} f(W^{k+1},H^{k+1})\in \partial_{\bw_i} F(W^{k+1}, H^{k+1})$.

Suppose $\|H^{k}(i,:)\|\neq 0$. If $\Phi_k(\bw_i^k) -\Phi_k(\bar{\bw}_i^k)\geq \frac{\delta_1}{2}\|\bw_i^k-\bar{\bw}_i^{k}\|^2$, then $\bw_i^{k+1}=\bar{\bw}_i^k$. By Lemma \ref{lem:wkbar} we have
\[
\bar{\bw}_i^{k}=\Pi_{\cp_m\cap\cq_m^{s_1}}\Big(\bw_i^k-\frac{1}{\|H^{k}(i,:)\|^2}\nabla\phi_k(\bw_i^k)\Big),
\]
which is a solution to the following minimization problem:
\[
\bw_i^{k+1}= \argmin_{\bw_i\in\Rm}\big\{\big\langle \bw_i-\bw_i^k, \nabla\phi_k(\bw_i^k)\big\rangle +\frac{1}{2\alpha_{ik}}\|\bw_i-\bw_i^k\|^2+\delta_{\cp_m}(\bw)+\delta_{\cq_m^{s_1}}(\bw)\big\},
\]
where $\alpha_{ik}=1/\|(H^{k}(i,:))^T \|^2$. Then there exists an element $\xi_i^{k+1}\in\partial\delta_{\cp_m}(\bw_i^{k+1})$ and $\zeta_i^{k+1}$ $ \in \partial\delta_{\cq_m^{s_1}}(\bw_i^{k+1})$ such that
\[
\nabla\phi_k(\bw_i^k)+\frac{1}{\alpha_{ik}}(\bw_i^{k+1}-\bw_i^k)+\xi_i^{k+1}+\zeta_i^{k+1}=\bf 0.
\]
Therefore, ${\ba}_i^{k+1}=\nabla_{\bw_i} f(W^{k+1},H^{k+1})+\xi_i^{k+1} +\zeta_i^{k+1}\in  \partial_{\bw_i} F(W^{k+1}, H^{k+1})$.

On the other hand, if $\Phi_{k}(\bw_i^k) -\Phi_{k}(\bar{\bw}_i^k)< \frac{\delta_1}{2}\|\bw_i^k-\bar{\bw}_i^{k}\|^2$, then we have by \eqref{eqw2},
\[
\bw_i^{k+1} = \Pi_{\cp_m\cap\cq_m^{s_1}} \Big(\bw_i^{k}-\mu_{ik} \nabla\phi_{k}(\bw_i^k)\Big),
\]
which is a solution to the following minimization problem:
\[
\bw_i^{k+1}= \argmin_{\bw_i\in\Rm}\big\{\big\langle \bw_i-\bw_i^k, \nabla\phi_k(\bw_i^k)\big\rangle+\frac{1}{2\alpha_{ik}}\|\bw_i-\bw_i^k\|^2 +\delta_{\cp_m}(\bw)+\delta_{\cq_m^{s_1}}(\bw)\big\},
\]
where  $\alpha_{ik}=\mu_{ik}$. Then there exists an element $\vartheta_i^{k+1}\in\partial\delta_{\cp_m}(\bw_i^{k+1})$ and $\varsigma_i^{k+1}\in\partial\delta_{\cq_m^{s_1}}(\bw_i^{k+1})$ such that
\[
\nabla\phi_k(\bw_i^k)+\frac{1}{\alpha_{ik}}(\bw_i^{k+1}-\bw_i^k)+\vartheta_i^{k+1} +\varsigma_i^{k+1}=\bf 0.
\]
Therefore, ${\ba}_i^{k+1}=\nabla_{\bw_i} f(W^{k+1},H^{k+1})+\vartheta_i^{k+1} +\varsigma_i^{k+1} \in  \partial_{\bw_i} F(W^{k+1}, H^{k+1})$.

Next, we show that $B^{k+1} \in \partial_H F(W^{k+1}, H^{k+1})$. For any fixed $1\le t\le n$, it is easy to see that  the iterate $\bh_t^{k+1}$ defined in Algorithm \ref{alg1} is a solution to the following minimization problem:
\[
\min\limits_{\bh_t\in\Rr} \big\langle \bh_t-\bh_t^{k}, \nabla\psi_k(\bh_{t}^k)\big\rangle+\frac{1}{2\nu_{tk}}\|\bh_t-\bh_t^{k}\|^2+\delta_{\cp_r}(\bh_t)+\delta_{\cq_r^{s_2}}(\bh_t).
\]
Then there exists an element $\theta_t^{k+1}\in\partial \delta_{\cp_r}(\bh_t^{k+1})$ and $\eta_t^{k+1}\in\partial\delta_{\cq_r^{s_2}}(\bh_t^{k+1})$ such that
\[
\nabla\psi_k(\bh_t^{k})+\frac{1}{\nu_{tk}}(\bh_t^{k+1}-\bh_t^{k})+\theta_t^{k+1}+\eta_t^{k+1}=\bf 0.
\]
Hence, $\bb_t^{k+1}=\nabla_{\bh_t} f(W^{k+1},H^{k+1})+\theta_t^{k+1} +\eta_t^{k+1}\in  \partial_{\bh_t} F(W^{k+1}, H^{k+1})$.

Finally, we show that  \eqref{bd:abk} holds.
By the definition of $\ba_i^{k+1}$ and $\phi_k$ and
using  \eqref{f:grad} we have, for any $1\le i\le r$,
\[
\ba_i^{k+1}=\bu_i^{k+1}+\bg_i^{k+1},
\]
where $\bu_i^{k+1}=W^{k+1}\big(H^{k+1}(H^{k+1}(i,:))^T-H^k(H^k(i,:))^T\big)+V\big((H^k(i,:))^T-(H^{k+1}(i,:))^T\big)$ and $\bg_i^{k+1}=\big(\sum_{j=i}^{r}(\bw_j^{k+1}-\bw_j^k)H^{k}(j,:)\big)(H^{k}(i,:))^T
-  \frac{1}{\alpha_{ik}}(\bw_i^{k+1}-\bw_i^k)$. Thus,
\[
A^{k+1}=U^{k+1}+G^{k+1},\quad U^{k+1}:=[\bu_1^{k+1}, \ldots, \bu_r^{k+1}],\quad G^{k+1}:=[\bg_1^{k+1}, \ldots, \bg_r^{k+1}].
\]
By the definition of $\ba_i^{k+1}$ we obtain
\begin{eqnarray*}
\|\bg_i^{k+1}\|&\le&\Big(\sum\limits_{j=i}^{r}\|\bw_j^{k+1}-\bw_j^k\| \|H^{k}(j,:)\|\Big)\|H^{k}(i,:)\|
+ (\|H^{k}(i,:)\|^2+\delta_1)\|\bw_i^{k+1}-\bw_i^k\| \\
&\le& (2n+\delta_1) \sum\limits_{j=i}^{r}\|\bw_i^{k+1}-\bw_i^k\|_F,
\end{eqnarray*}
where the second inequality uses the fact that   $\|H^{k}(i,:)\|^2\leq n$. Hence,
\begin{eqnarray*} \label{ineqforw2}
\|G^{k+1}\|_F^2 & =& \sum\limits_{i=1}^{r}\|\bg_i^{k+1}\|^2
\le \sum\limits_{i=1}^{r} (2n+\delta_1)^2 \Big(\sum\limits_{j=i}^{r}\|\bw_j^{k+1}-\bw_j^k\|\Big)^2  \nonumber\\
&\le& (2n+\delta_1)^2 \sum\limits_{i=1}^{r} (r-i+1)\sum\limits_{j=i}^{r}\|\bw_j^{k+1}-\bw_j^k\|^2  \nonumber\\
&\le& \frac{1}{2}r(r+1) (2n+\delta_1)^2 \|W^{k+1}-W^k\|_F^2.
\end{eqnarray*}
In addition, by the  definition of $\bu_i^{k+1}$ we have
\begin{eqnarray*} \label{ineqforw1}
\|U^{k+1}\|_F & =& \|W^{k+1}\big(H^{k+1}(H^{k+1})^T-H^k(H^k)^T\big) +V(H^k-H^{k+1})^T\|_F\nonumber\\
&\le&\|W^{k+1}\|_2\|H^{k+1}\|_2\|H^{k+1}-H^{k}\|_F\nonumber\\
&&+\|W^{k+1}\|_2\|H^{k+1}-H^{k}\|_F\|H^{k}\|_2+\|V\|_2\|H^{k+1}-H^k\|_F\nonumber\\
&\le& (2\sqrt{nr}+\|V\|_2) \|H^{k+1}-H^k\|_F,
\end{eqnarray*}
where the second inequality uses the fact that   $\|W^{k+1}\|_2^2\leq r$ and $\|H^{k}\|_2^2\le n$.
Thus,
\begin{eqnarray} \label{ineq:bk1}
\|A^{k+1}\|_F^2 & \le&  2\|U^{k+1}\|_F^2+2\|G^{k+1}\|_F^2\nonumber\\
&\le& 2(2\sqrt{nr}+\|V\|_2)^2 \|H^{k+1}-H^k\|_F^2+r(r+1)(2n+\delta_1)^2 \|W^{k+1}-W^k\|_F^2.
\end{eqnarray}

By the definition of $B^{k+1}$ and $\psi_k$ we have
\[
B^{k+1}=\nabla_{H} f(W^{k+1},H^{k+1})-\nabla_{H} f(W^{k+1},H^{k})-D^k(H^{k+1}-H^k),
\]
where $D^k=\diag(1/\nu_{1k},\ldots,1/\nu_{nk})$.
By the definition of $\nu_{tk}$ and using  \eqref{mu:bd1} and \eqref{mu:bd2} we have
\BE\label{dk:bd}
\|W^k\|_2^2\le r\quad\mbox{and}\quad\|D^k\|_2 \leq r+\delta_2.
\EE
From \eqref{h:2n} and \eqref{dk:bd} we have
\begin{eqnarray}\label{ineq:ak1}
\|B^{k+1}\|_F &\le& \|\nabla_H f(W^{k+1},H^{k+1})-\nabla_H f(W^{k+1},H^{k})\|_F+\|D^k(H^{k+1}-H^k)\|_F \nonumber\\
&\leq&\|W^{k+1}\|_2^2\|H^{k+1}-H^k\|_F+\|D^k\|_2\|H^{k+1}-H^k\|_F \nonumber\\
& \leq&(\delta_2+2r)\|H^{k+1}-H^k\|_F.
\end{eqnarray}
Therefore, \eqref{bd:abk}  follows from  \eqref{ineq:bk1} and \eqref{ineq:ak1}.
\end{proof}

In the following, we set $\cl(W^0, H^0)$ to be the set of all accumulation points of the sequence $\{(W^k, H^k)\}$ generated by Algorithm \ref{alg1}.
Regarding the set $\cl(W^0, H^0)$, we have the following result.

\begin{lemma}\label{lem:ly0}
Let $\{(W^k, H^k)\}$  be the sequence generated by Algorithm \ref{alg1}. Then
$\cl(W^0, H^0)$ is a nonempty and compact set and the function $F(W, H)$ is finite and constant on $\cl(W^0, H^0)$.
\end{lemma}
\begin{proof}
It is obvious that $\cl(W^0, H^0)$ is nonempty and compact since $\{(W^k, H^k)\}$ is bounded.
By Lemma \ref{lem:fkn} we know that $\{F(W^k, H^k)\}$ converges to a finite
limit $F_*$.
For any $(W^*, H^*)\in\cl(W^0, H^0)$, there exists a subsequence $\{(W^{k_q}, H^{k_q})\}$ such that $\lim_{q\to\infty}(W^{k_t}, H^{k_t})=(W^*, H^*)$. Thus,
$F(W^*, H^*)=\lim_{q\to\infty}F(W^{k_q}, H^{k_q})=F_*$.
Hence, $F$ is finite and constant on $\cl(W^0, H^0)$.
\end{proof}

On the global convergence of Algorithm \ref{alg1}, we have the following result.
\begin{theorem}\label{thm:gc}
Let  $\{(W^k, H^k)\}$  be the sequence generated by Algorithm \ref{alg1}. Then every accumulation point of $\{(W^k, H^k)\}$ is a critical point of $F$.
\end{theorem}
\begin{proof} The proof is similar to that of \cite[Lemma 5(i)]{BS14} by using Remark \ref{rem:fkn} and  Lemmas \ref{lem:abk}--\ref{lem:ly0}.
\end{proof}

Finally, on the convergence of the sequence $\{(W^k, H^k)\}$ generated by Algorithm \ref{alg1}, we have the following result. The proof follows from the arguments  similar to that of \cite[Theorem 1]{BS14} by using Lemmas
\ref{lem:fwhkn}--\ref{lem:fkn}, Lemmas \ref{lem:abk}--\ref{lem:ly0}, and  Theorem \ref{thm:gc}. Hence, we omit it here.
\begin{theorem}\label{thm:whk-conv}
Let  $\{(W^k, H^k)\}$  be the sequence generated by Algorithm \ref{alg1}. Then the sequence $\{(W^k, H^k)\}$ converges to a critical point of $F$.
\end{theorem}

The following theorem shows that the sequence $\{(W^k, H^k)\}$  generated by Algorithm \ref{alg1} converges to  a special critical point of $F$.
\begin{theorem}\label{thm:cp}
Let  $\{(W^k, H^k)\}$  be the sequence generated by Algorithm \ref{alg1}, which converges to  $(W^*, H^*)$. Then we have, for any $1\le i\le r$,
\BE\label{bdwres}
F(W^*,H^*)\le F( \bw_1^{*}, \ldots, \bw_{i-1}^{*}, \bw_{i}, \bw_{i+1}^*, \ldots, \bw_r^*, H^*)
+\delta_1,
\EE
for all $\bw_i\in\cp_m\cap\cq_m^{s_1}$.
\end{theorem}
\begin{proof}
Let $1\le i\le r$ be fixed. If  $\|H^{k}(i,:)\|=0$, then, by Step 1 of Algorithm \ref{alg1} we have $\bw_i^{k+1}=\bw_i^{k}$ and
\begin{eqnarray}\label{ineq:fht-1}
&&F(\bw_1^{k+1}, \ldots, \bw_{i-1}^{k+1}, \bw_{i}^{k+1}, \bw_{i+1}^k, \ldots, \bw_r^k, H^k) \nonumber\\
&=& F(\bw_1^{k+1}, \ldots, \bw_{i-1}^{k+1}, \bw_{i}, \bw_{i+1}^k, \ldots, \bw_r^k, H^k),
\end{eqnarray}
for all $\bw_i\in\cp_m\cap\cq_m^{s_1}$.

Suppose $\|H^{k}(i,:)\|\neq 0$. If $\Phi_k(\bw_i^k) -\Phi_k(\bar{\bw}_i^k)\geq \frac{\delta_1}{2}\|\bw_i^k-\bar{\bw}^{k}\|^2$, then
\[
\bw_i^{k+1}=\bar{\bw}_i^k=\argmin_{\bw_i\in\Rm}F( \bw_1^{k+1}, \ldots, \bw_{i-1}^{k+1}, \bw_{i}, \bw_{i+1}^k, \ldots, \bw_r^k, H^k).
\]
Thus,
\begin{eqnarray}\label{ineq:fht-2}
&&F(\bw_1^{k+1}, \ldots, \bw_{i-1}^{k+1}, \bw_{i}^{k+1}, \bw_{i+1}^k, \ldots, \bw_r^k, H^k)  \nonumber\\
&\leq& F(\bw_1^{k+1}, \ldots, \bw_{i-1}^{k+1}, \bw_{i}, \bw_{i+1}^k, \ldots, \bw_r^k, H^k),
\end{eqnarray}
for all $\bw_i\in\cp_m\cap\cq_m^{s_1}$.

On the other hand, if $\Phi_{k}(\bw_i^k) -\Phi_{k}(\bar{\bw}_i^k)< \frac{\delta_1}{2}\|\bw_i^k-\bar{\bw}_i^{k}\|^2$, then  $\bw_i^{k+1}$ is determined by  \eqref{eqw2}, i.e.,
\[
\bw_i^{k+1}= \argmin_{\bw_i\in \cp_m\cap\cq_m^{s_1}}\big\{\big\langle \bw_i-\bw_i^k, \nabla\phi_k(\bw_i^k)\big\rangle+\frac{1}{2\mu_{ik}}\|\bw_i-\bw_i^k\|^2\big\}.
\]
We note that $\phi_k$ is a quadratic function. By simple calculation, we find that, for any $\bw_i\in\cp_m\cap\cq_m^{s_1}$,
\[
\phi_k(\bw_i)=\phi_k(\bw_i^k)+\big\langle \bw_i-\bw_i^k, \nabla\phi_k(\bw_i^k)\big\rangle+\frac{1}{2}\|H^{k}(i,:)\|^2\|\bw_i-\bw_i^k\|^2.
\]
Thus,
\begin{eqnarray*}
&&\phi_k(\bw_i^{k+1})=\phi_k(\bw_i^k)+\big\langle \bw_i^{k+1}-\bw_i^k, \nabla\phi_k(\bw_i^k)\big\rangle+\frac{1}{2}\|H^{k}(i,:)\|^2\|\bw_i^{k+1}-\bw_i^k\|^2 \\
&\le& \phi_k(\bw_i^k)+\big\langle \bw_i-\bw_i^k, \nabla\phi_k(\bw_i^k)\big\rangle+\frac{1}{2\mu_{ik}}\|\bw_i-\bw_i^k\|^2-\frac{\delta_1}{2} \|\bw_i^{k+1}-\bw_i^k\|^2\\
&=&\phi_k(\bw_i)+\frac{\delta_1}{2}\|\bw_i-\bw_i^k\|^2-\frac{\delta_1}{2} \|\bw_i^{k+1}-\bw_i^k\|^2,
\end{eqnarray*}
for all $\bw_i\in\cp_m\cap\cq_m^{s_1}$. This means that
\begin{eqnarray}\label{ineq:fht-3}
&&F( \bw_1^{k+1}, \ldots, \bw_{i-1}^{k+1}, \bw_{i}^{k+1}, \bw_{i+1}^k, \ldots, \bw_r^k, H^k)\nonumber \\
&\le&  F( \bw_1^{k+1}, \ldots, \bw_{i-1}^{k+1}, \bw_{i}, \bw_{i+1}^k, \ldots, \bw_r^k, H^k) \nonumber \\
&&+\frac{\delta_1}{2}\|\bw_i-\bw_i^k\|^2-\frac{\delta_1}{2} \|\bw_i^{k+1}-\bw_i^k\|^2,
\end{eqnarray}
for all $\bw_i\in\cp_m\cap\cq_m^{s_1}$.
By hypothesis, $\{(W^k, H^k)\}$ converges to  $(W^*, H^*)$. Then, by using the arguments similar to that of Theorem \ref{thm:gc} to the inequalities \eqref{ineq:fht-1}--\eqref{ineq:fht-3}, we have
\begin{eqnarray*}
&&F(\bw_1^{*}, \ldots, \bw_{i-1}^{*}, \bw_{i}^{*}, \bw_{i+1}^*, \ldots, \bw_r^*, H^{*}) \\
&\le& F( \bw_1^{*}, \ldots, \bw_{i-1}^{*}, \bw_{i}, \bw_{i+1}^*, \ldots, \bw_r^*, H^{*}) +\frac{\delta_1}{2}\|\bw_i-\bw_i^*\|^2,
\end{eqnarray*}
for all $\bw_i\in\cp_m\cap\cq_m^{s_1}$. Therefore,  \eqref{bdwres} holds since  $\|\bw_i-\bw_i^*\|^2\le \|\bw_i\|^2+\|\bw_i^*\|^2\le 2$  for all $\bw_i,\bw_i^*\in\cp_m\cap\cq_m^{s_1}$.
\end{proof}
\begin{theorem}\label{thm:convex}
Let  $\{(W^k, H^k)\}$  be the sequence generated by Algorithm \ref{alg1}, which converges to  $(W^*, H^*)$. If $s_2=r$, then we have
\BE\label{bd:fw}
F(W^*,H^*)\leq F(W^*,H) \quad\forall H\in\cg_3.
\EE
\end{theorem}
\begin{proof}
Let $1\le t\le n$ be fixed. From Step 2 of Algorithm \ref{alg1} we see that  $\bh_t^{k+1}\in\cp_r$ solves the following minimization problem:
\BE\label{op:wik}
\begin{array}{ll}
\min\limits_{\bh_t\in\Rr}  &  \displaystyle \frac{1}{2\nu_{tk}} \big\|\bh_t-\big(\bh_t^k-\nu_{tk}\nabla\psi_k(\bh_t^k)\big)\big\|^2 \\[2mm]
\mbox{s.t.} &{\bf 1}_r^T\bh_t=1,\quad \bh_t\ge 0,
\end{array}
\EE
where ${\bf 1}_r\in\R^r$ is a vector of all ones,
Then there exist two Lagrange multipliers $\gamma_{k+1}$ and $\bp^{k+1}$ such that the following first-order optimization conditions hold:
\BE\label{kkt:wik}
\left\{
\begin{array}{c}
\frac{1}{\nu_{tk}} (\bh_t^{k+1}-\bh_t^k)+\nabla\psi_k(\bh_t^k)+\gamma_{k+1}{\bf 1}_r- \bp^{k+1}=0,\\[2mm]
{\bf 1}_r^T\bh_t^{k+1}=1,\quad \bh_t^{k+1}\geq0,\\[2mm]
\bp^{k+1}\ge 0,\quad \langle \bp^{k+1},\bh_t^{k+1}\rangle=0.
\end{array}
\right.
\EE
We can obtain that
\begin{eqnarray*}
\left\{
\begin{array}{lll}
\gamma_{k+1} &=&-\langle\bh_t^{k+1},1/\nu_{tk}(\bh_t^{k+1}-\bh_t^k) +\nabla\psi_k(\bh_t^k)\rangle, \\
\bp^{k+1} &=& \frac{1}{\nu_{tk}} (\bh_t^{k+1}-\bh_t^k)+\nabla\psi_k(\bh_t^k)+\gamma_{k+1}{\bf 1}_r.
\end{array}
\right.
\end{eqnarray*}
We note that $\nabla\psi_k(\bh_t^k)=(W^{k+1})^T(W^{k+1}\bh_t^k-\bv_t)$. By hypothesis, $\{(W^k, H^k)\}$ converges to  $(W^*, H^*)$. From Remark \ref{rem:fkn} we know that $\|H^{k+1}-H^k\|_F\to 0$ as $k\to \infty$.
Also, $\{\nu_{tk}\}$ is bounded. Taking $k\to\infty$ yields
\begin{eqnarray*}
\left\{
\begin{array}{lll}
\displaystyle \lim_{k\to\infty}\gamma_{k+1} &=&-\langle\bh_t^{*},(W^*)^T(W^*\bh_t^*-\bv_t)\rangle\equiv\gamma_*, \\
\displaystyle \lim_{k\to\infty}\bp^{k+1} &=& (W^*)^T(W^*\bh_t^*-\bv_t)+\gamma_{*}{\bf 1}_r\equiv \bp^*.
\end{array}
\right.
\end{eqnarray*}
This, together with  \eqref{kkt:wik}, implies that
\BE\label{kkt:wistar}
\left\{
\begin{array}{c}
(W^*)^T(W^*\bh_t^*-\bv_t)+\gamma_{*}{\bf 1}_r- \bp^{*}={\bf 0},\\[2mm]
{\bf 1}_r^T\bh_t^{*}=1,\quad \bh_t^{*}\geq0,\\[2mm]
\bp^{*}\ge 0,\quad \langle \bp^{*},\bh_t^{*}\rangle=0.
\end{array}
\right.
\EE
It is obvious that $\bh_t^*$ satisfies \eqref{kkt:wistar} for some  Lagrange multipliers $\gamma_{*}$ and $\bp^{*}$.
Thus,  $\bh_t^*$ is a global solution of the following minimization problem:
\BE\label{op:wi}
\begin{array}{ll}
\min\limits_{\bw_i\in\Rr}  &  \displaystyle \frac{1}{2} \big\|W^*\bh_t-\bv_t\big\|^2 \\[2mm]
\mbox{s.t.} & {\bf 1}_r^T\bh_t=1,\quad \bh_t\ge 0.
\end{array}
\EE
We note that
\[
f(W,H^*)= \frac{1}{2} \|V-W^*H\|_F^2=\sum_{t=1}^{n}\frac{1}{2} \big\|W^*\bh_t-\bv_t\big\|^2
\]
and $H\in\cg_3$ if and only if ${\bf 1}_r^T\bh_t=1$, $\bh_t\ge 0$ ($t=1,\ldots,n$)
for all $H:=[\bh_1,\ldots,\bh_n]\in\Rrn$.
It is easy to see that $H^*:=[\bh_1^*,\ldots,\bh_n^*]$ is a global solution of the following minimization problem:
\BE\label{op:w}
\begin{array}{ll}
\min\limits_{H\in\Rrn}  &  \displaystyle f(W^*,H)= \frac{1}{2} \|V-W^*H\|_F^2 \\[2mm]
\mbox{s.t.} & H\in\cg_3.
\end{array}
\EE
Therefore, $F(W^*,H^*)\leq F(W^*,H)$ for all $H \in \cg_3$.
\end{proof}

\begin{remark}
From Theorems \ref{thm:cp} and \ref{thm:convex}, we observe that the sequence $\{(W^k, H^k)\}$  generated by Algorithm \ref{alg1} converges to  a special critical point of $F$, which is nearly a global minimum of $F$ over each column vector of  the $W$-factor if the parameter $\delta_1>0$ is  sufficiently  small and is globally minimized over the $H$-factor as a whole if the sparseness constraint of $H$ is removed.
In the latter numerical tests, one can see that the solution to the SSMF obtained by the proposed algorithm may have a smaller factorization residual than the PALM method.
\end{remark}

\section{Numerical experiments}\label{sec4}

In this section, we report the numerical performance of  Algorithm \ref{alg1}  for solving the SSMF \eqref{pro:smf-eq}  over synthetic and real data.
To illustrate the effectiveness of our method, we compare Algorithm \ref{alg1} with the PALM algorithm in \cite{BS14} (i.e., Algorithm \ref{PALM}),  the PLSA   with additive sparsing regularization for topic models (ARTMsparse) in \cite{VP14}, which was implemented  by simple modification of the PLSA algorithm \cite{H99}\footnote{\url{https://github.com/lizhangzhan/plsa}}. All the  numerical tests were carried out in {\tt MATLAB R2020a} on a linux server with an  Intel Xeon CPU Gold 6230 of 2.10 GHz and 32 GB of RAM.

In our numerical tests, for Algorithms \ref{PALM} and \ref{alg1}, we set  $\delta_1=\delta_2=10^{-6}$, and $c=1/\delta_2=10^{6}$. The errors between the restored distributions $V_{ij}^k=(W^kH^k)_{ij}$ and the model ones $V_{ij}$ were measured by the averaged Hellinger distance \cite{VP14}:
\[
H(V, V^k)=\frac{1}{m}\sum\limits_{j=1}^m\Big(\frac{1}{2}\sum\limits_{i=1}^n( \sqrt{V_{ij}^k}-\sqrt{V_{ij}})^2\Big)^{\frac{1}{2}}.
\]
For comparison purposes, we use the symbols  `{\tt ITmax}', `{\tt ct.}',  `{\tt res.}', `{\tt nnz1.}',
and `{\tt nnz2.}' to denote the largest number of iterations, the averaged total computing time in seconds,  the averaged relative residual $\|V-W^kH^k\|_F/\|V\|_F$, and the averaged  numbers of non-zero elements of $W^k$ and $H^k$ at the final iterates of the corresponding algorithms, respectively.

\subsection{Synthetic data}
We first compare the sparsity of the results obtained by several algorithms on synthetic data and the change of the difference between restored $V^*$ and model $V$ with the number of iteration steps.
\begin{example}\label{synex1}
In this example, we consider the SSMF with fixed $(m,r,n)$ and sparsity.
Here, $V=\widehat{W}\widehat{H}\in \R^{1000\times 500}$ with $\widehat{W}=\Pi_{\cg_1}(\overline{W})\in\R^{1000\times 60}$  and $\widehat{H}=\Pi_{\cg_2}(\overline{H})\in\R^{60\times 500}$ with the true sparsity of each column of $\widehat{W}$ being ${\tt ts_1}=200$ and the true sparsity of each column of $\widehat{H}$ being ${\tt ts_2}=12$, where $\overline{W}$ and $\overline{H}$ are randomly generated by using {\tt rand}.  We report our numerical results for the prescribed sparsity $(s_1,s_2)=(230,17)$.
\end{example}

The numerical results for Example \ref{synex1} are given in Table \ref{syntable1}  by running Algorithms \ref{PALM}, \ref{alg1}, {PLSA} and {ARTMsparse} over $20$ randomly generated initial points with {\tt ITmax}=$300$. Figure \ref{synfig1}  displays the averaged quality of recovery versus the number of iterations for Example \ref{synex1}. We find from Table \ref{syntable1} and Figure \ref{synfig1}  that Algorithm \ref{alg1} works much better than the other methods in terms of the reconstruction error. In addition, the computed solutions via Algorithms \ref{PALM} and  \ref{alg1} have almost the same sparsity since they have the same sparsity level requirements.
\begin{table}[!ht]\renewcommand{\arraystretch}{1.0} \addtolength{\tabcolsep}{1.0pt}
\begin{center}{\scriptsize
  \begin{tabular}[c]{|c|c|c|c|c|}     \hline
                               & {\tt nnz1.}      & {\tt nnz2.} &${\tt H(V,W^kH^k)} $ & {\tt ct.}\\ \hline
{ Alg. \ref{PALM}}     &$17956$   &      $5511    $ &$0.3646 $  &$128.95$ \\ \hline
{Alg. \ref{alg1}}         &$12116$   &      $5751$     &$ 0.0107$  &$191.45$\\ \hline
{ PLSA}                    &$60000$  &       $30000$   & $ 0.5402$ &$507.24$\\ \hline
{ARTMsparse}      &$60000$  &       $2458$   & $0.5439 $       &$505.91$\\ \hline
\end{tabular}}
\end{center}
\caption{Numerical results for Example \ref{synex1}.} \label{syntable1}
\end{table}
\begin{figure}[!ht]
 \centering
\includegraphics[width=0.65\textwidth]{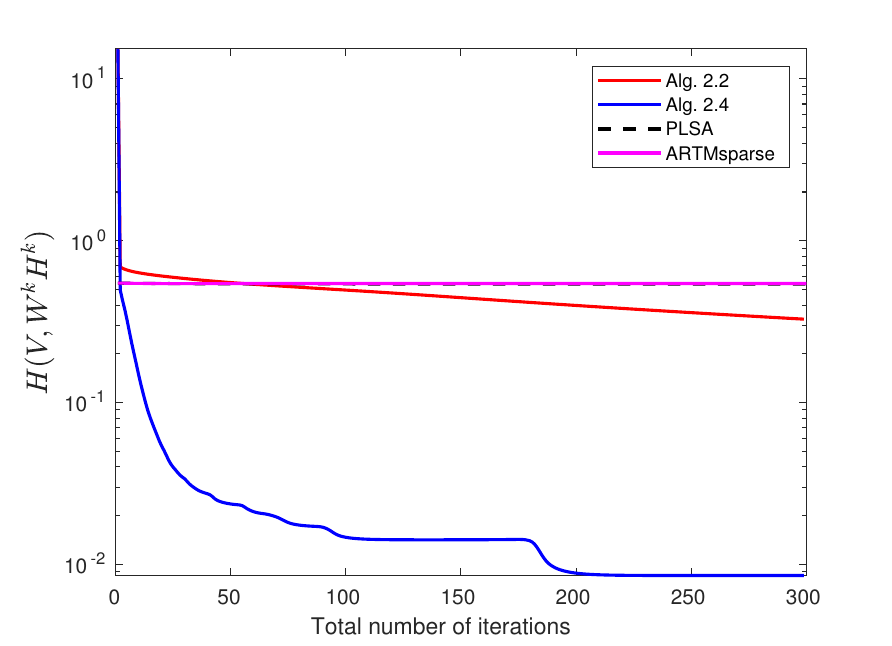}
 \caption{Reconstruction errors for Example \ref{synex1}.}\label{synfig1}
 \end{figure}
\begin{example}\label{synex:1}
We consider the SSMF with fixed $(m,r,n)$.
Here, $V=\widehat{W}\widehat{H}\in \R^{400\times 200}$ with $\widehat{W}=\Pi_{\cg_1}(\overline{W})\in\R^{400\times 100}$  and $\widehat{H}=\Pi_{\cg_2}(\overline{H})\in\R^{100\times 200}$ with the true sparsity of each column of $\widehat{W}$ being ${\tt ts_1}$ and the true sparsity of each column of $\widehat{H}$ being ${\tt ts_2}$, where $\overline{W}$ and $\overline{H}$ are randomly generated by using {\tt rand}.  We report our numerical results for the true sparsity {\tt ts}$_1=80, 100, 120, 140, 160$, {\tt ts}$_2={\tt ts}_1/4$, and prescribed sparsity $(s_1,s_2)=({\tt ts}_1+30,{\tt ts}_2+15)$.
\end{example}
\begin{example}\label{synex:2}
We consider the SSMF  with fixed $(m,r,n)$ and different prescribed  sparsity. Here, $V=\widehat{W}\widehat{H}\in \R^{400\times 200}$, where $\widehat{W}=\Pi_{\cg_1}(\overline{W})\in\R^{400\times 50}$  and $\widehat{H}=\Pi_{\cg_2}(\overline{H})\in\R^{50\times 200}$ with the true sparsity of each column of $\widehat{W}$ being ${\tt ts_1}=80$ and the true sparsity of each column of $\widehat{H}$ being ${\tt ts_1}=10$, where $\overline{W}$ and $\overline{H}$ are randomly generated by using {\tt rand}.
We report our numerical results for fixed $V$ and the prescribed sparsity {\rm (a)} $(s_1,s_2)=(80,10)$; {\rm (b)} $(s_1,s_2)=(96, 12)$; {\rm (c)} $(s_1,s_2)=(112, 14)$; {\rm (d)} $(s_1,s_2)=(128, 16)$; {\rm (e)} $(s_1,s_2)=(144, 18)$; {\rm (f)} $(s_1,s_2)=(160, 20)$.
\end{example}

For Examples \ref{synex:1}--\ref{synex:2}, we randomly choose  the same initial point $(W^0,H^0)$ and the stopping criterion is set to be
\[
\frac{\|W^kH^k-W^{k-1}H^{k-1}\|}{\|W^{k-1}H^{k-1}\|}\le{\tt tol},
\]
where ``{\tt tol}" is a prescribed tolerance.

Table \ref{Ptable1} displays the numerical results for Example \ref{synex:1} by running Algorithms \ref{PALM} and \ref{alg1} over $100$ randomly generated $V$ and initial points with ${\tt tol}=10^{-5}$ and ${\tt ITmax}=6000$ and for Example \ref{synex:2} by running Algorithms \ref{PALM} and \ref{alg1} over $100$ randomly generated initial points with ${\tt tol}=10^{-5}$ and ${\tt ITmax}=4000$. Here, the recovery is considered to be successful if the relative reconstruction error $\|V-W^{k}H^k\|_F/\|V\|_F$ is less than $1\%$ at the final iterate $(W^k,H^k)$. 

We can observe from Table \ref{Ptable1} that with the increase of prescribed sparsity, the probability of successful  reconstruction of  our algorithm  is much higher than Algorithm \ref{PALM}. 

\begin{table}[!ht]\renewcommand{\arraystretch}{1.0} \addtolength{\tabcolsep}{1.0pt}
\begin{center}{\scriptsize
  \begin{tabular}[c]{|c|c|c|c|c|c|}     \hline
 \multicolumn{3}{|c|}{Ex. \ref{synex:1}}& \multicolumn{3}{|c|}{Ex. \ref{synex:2}} \\ \hline
\multirow{2}{*}{{\tt ts}$_1$} & Alg. \ref{PALM} & Alg. \ref{alg1} & \multirow{2}{*}{$(s_1,s_2)$} & Alg. \ref{PALM} & Alg. \ref{alg1}    \\ \cline{2-3} \cline{5-6}
& probability       & probability && probability  & probability \\ \hline
{80} &$3\%$   & $59\%$ & (a)  & $0\%$ & $66\%$   \\ \hline
{100} &$2\%$  & $54\%$ & (b)  & $0\%$ & $84\%$   \\ \hline
{120} &$5\%$  & $60\%$ & (c)  & $0\%$ & $83\%$   \\ \hline
{140} &$2\%$  & $57\%$ & (d)  & $1\%$ & $85\%$   \\ \hline
{160} &$3\%$  & $55\%$ & (e)  & $7\%$ & $89\%$   \\ \hline
      &        &         & (f)  &$19\%$ & $88\%$ \\ \hline
\end{tabular}}
\end{center}
\caption{Numerical results for Examples \ref{synex:1} and \ref{synex:2}.} \label{Ptable1}

\end{table}



\subsection{Real data}
In this subsection,
we first consider a numerical example in document recognition \cite{SX13}.
\begin{example}\label{ex:5}
We consider the data converted from grayscale images from the MNIST Handwritten Digits data set\footnote{\url{http://yann.lecun.com/exdb/mnist/}}. We arbitrarily choose $800$  $20\times 20$ images of handwritten digit $3$, which are vectorized to column vectors and normalized to have total sum of one. Then we use the vectorized and normalized images to form the $400\times 800$ target matrix $V$. We would like to decompose $V$ as the product of an $400\times 196$ stochastic matrix $W$ and  a $196\times 800$ sparse stochastic matrix $H$.
\end{example}

Figures \ref{fig1}--\ref{fig2} give the decomposition results for Example \ref{ex:5} with $(s_1, s_2)=(100,100)$ and $(s_1, s_2)=(150,120)$, respectively, where ${\tt tol}=10^{-3}$ and  ${\tt ITmax}=5000$. Here, we only show $36$ reshaped columns of the factors as images because of space limits. Table \ref{table5} displays the relative reconstruction error and the total computing time and Figures \ref{fig3}--\ref{fig4} show the convergence curve and the computing time curve versus the number of iterations.

We see from  Figures \ref{fig1}--\ref{fig2} that our method  is more effective   than  Algorithm \ref{PALM} in  document recognition. We also observe from Table \ref{table5} and Figures \ref{fig3}--\ref{fig4} that, compared with Algorithm \ref{PALM}, our algorithm can obtain smaller factorization residual though we need more computing cost in each iteration.

\begin{figure}[!ht]
 \centering
\includegraphics[width=1.05\textwidth]{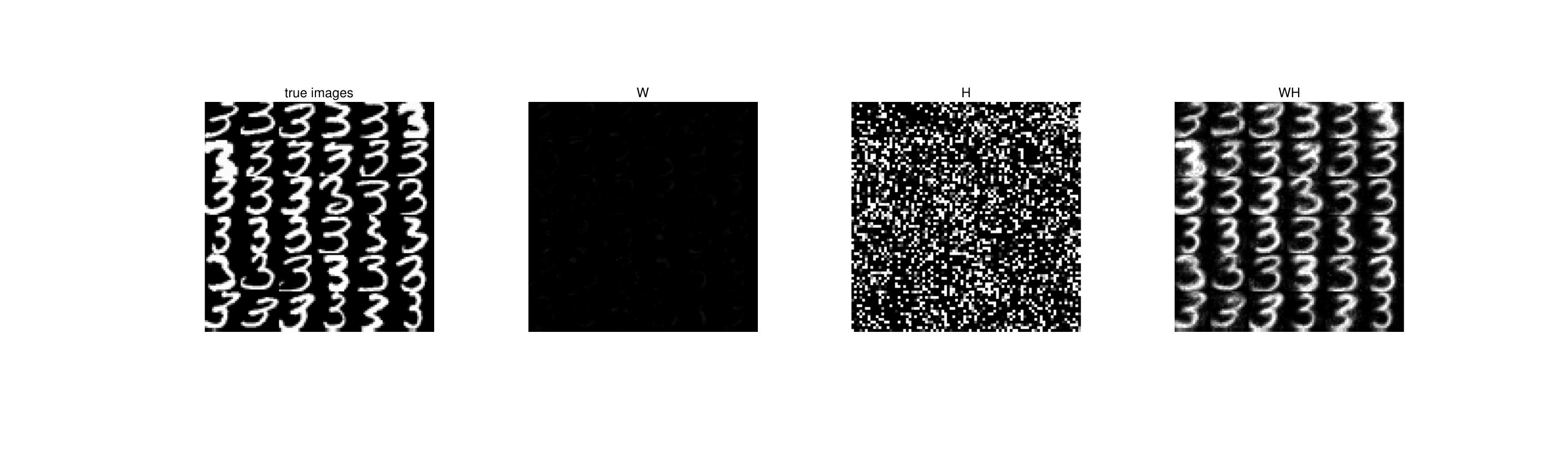}\\
\includegraphics[width=1.05\textwidth]{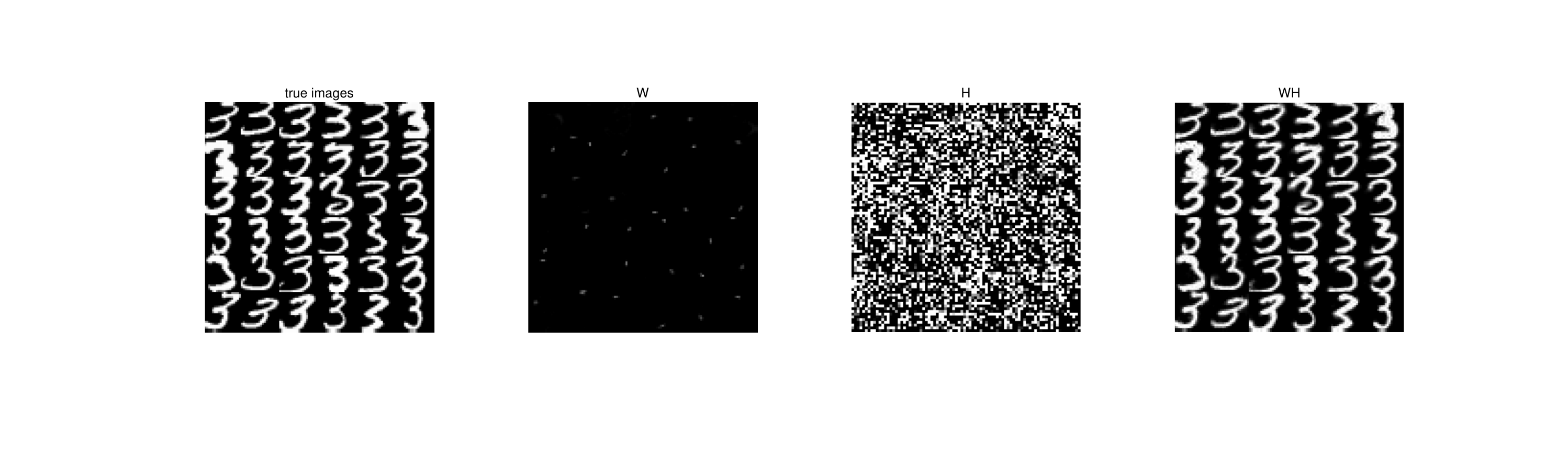}
 \caption{Constructed images with $(s_1, s_2)=(100,100)$ via Alg. \ref{PALM} (top) and Alg.  \ref{alg1} (bottom) for Example \ref{ex:5}.}\label{fig1}
 \end{figure}
\begin{figure}[!ht]
 \centering
\includegraphics[width=1.0\textwidth]{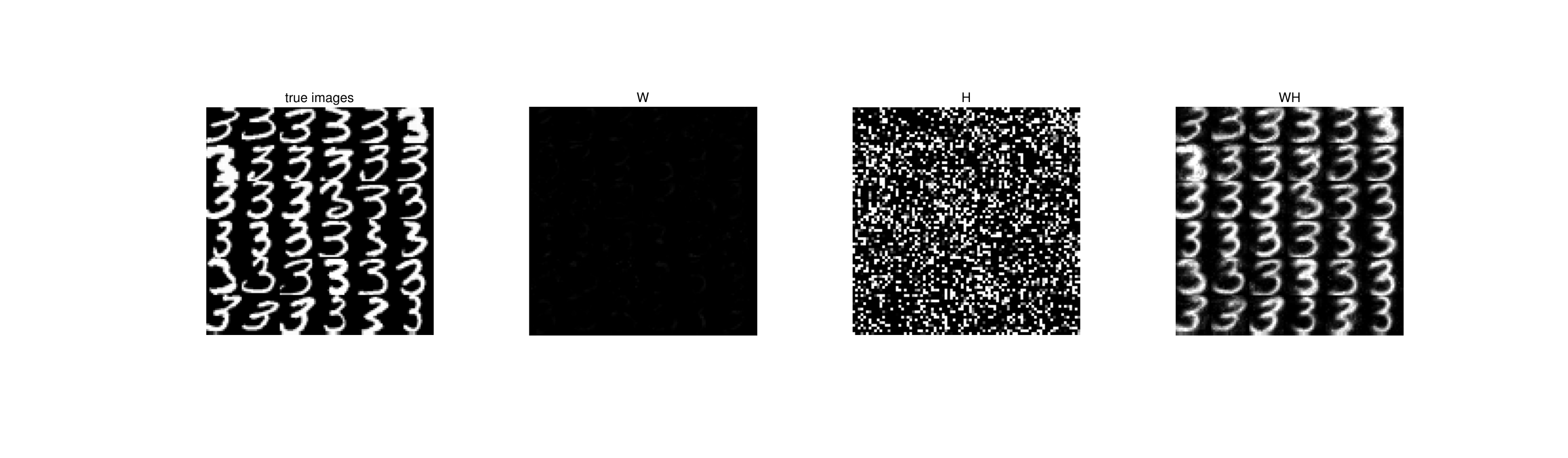}\\
\includegraphics[width=1.0\textwidth]{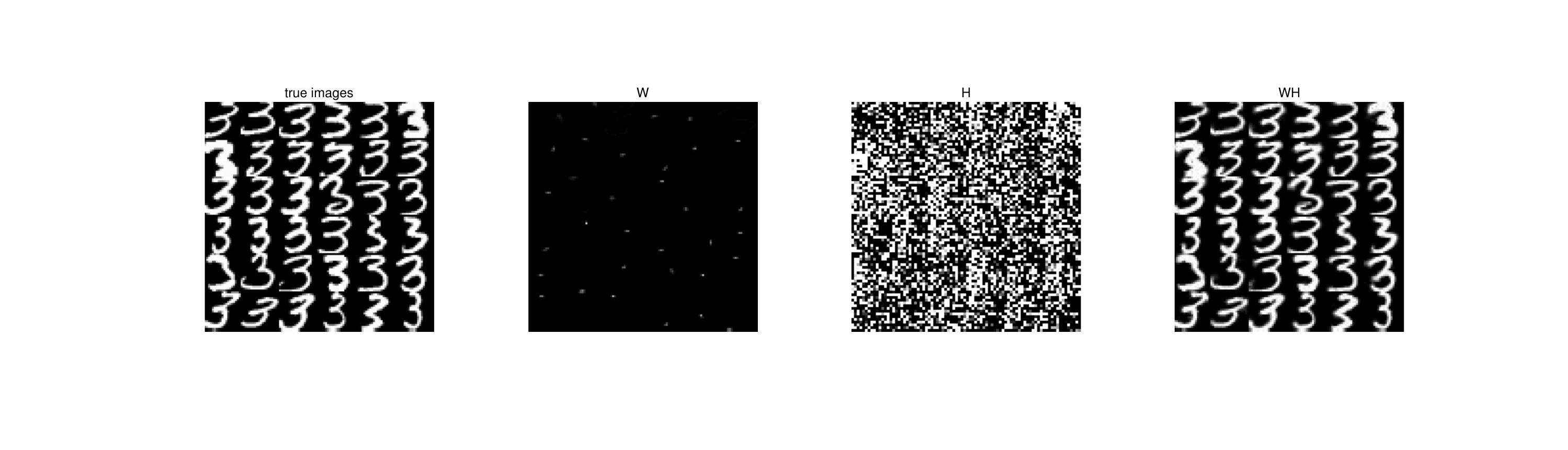}
 \caption{Constructed images with $(s_1, s_2)=(150,120)$ via Alg. \ref{PALM} (top) and Alg.  \ref{alg1} (bottom) for Example \ref{ex:5}.}\label{fig2}
\end{figure}

\begin{figure}[!ht]
 \centering
\includegraphics[width=1.0\textwidth]{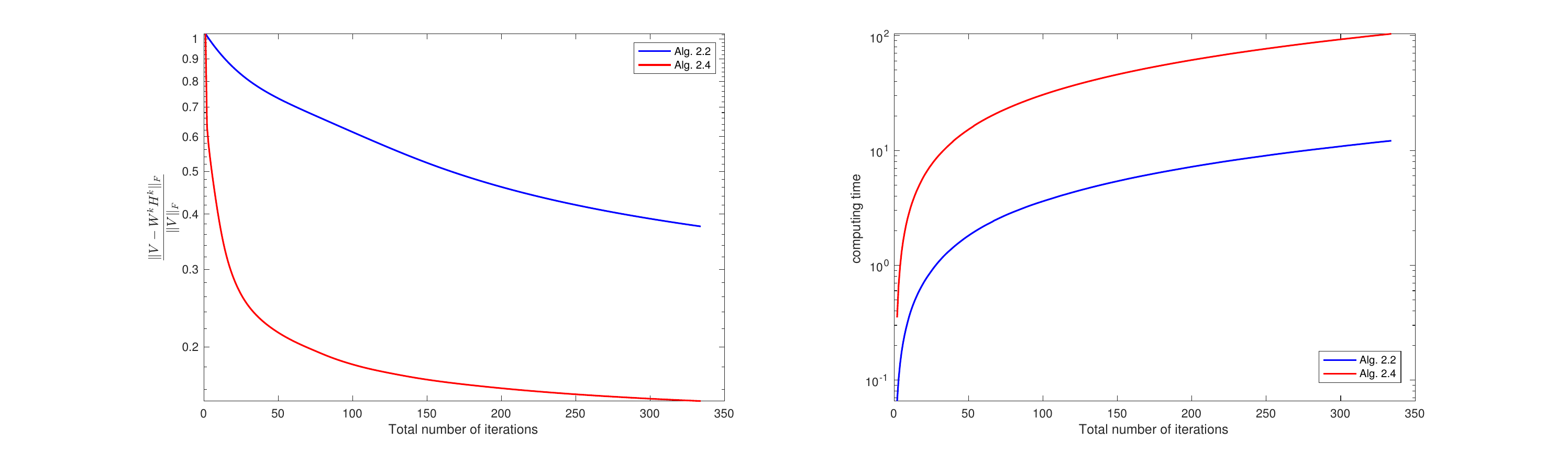}
 \caption{Convergence and computing time curves for one of the tests ($(s_1, s_2)=(100,100)$)  for Example \ref{ex:5}.}\label{fig3}
 \end{figure}

\begin{figure}[!ht]
 \centering
\includegraphics[width=1.0\textwidth]{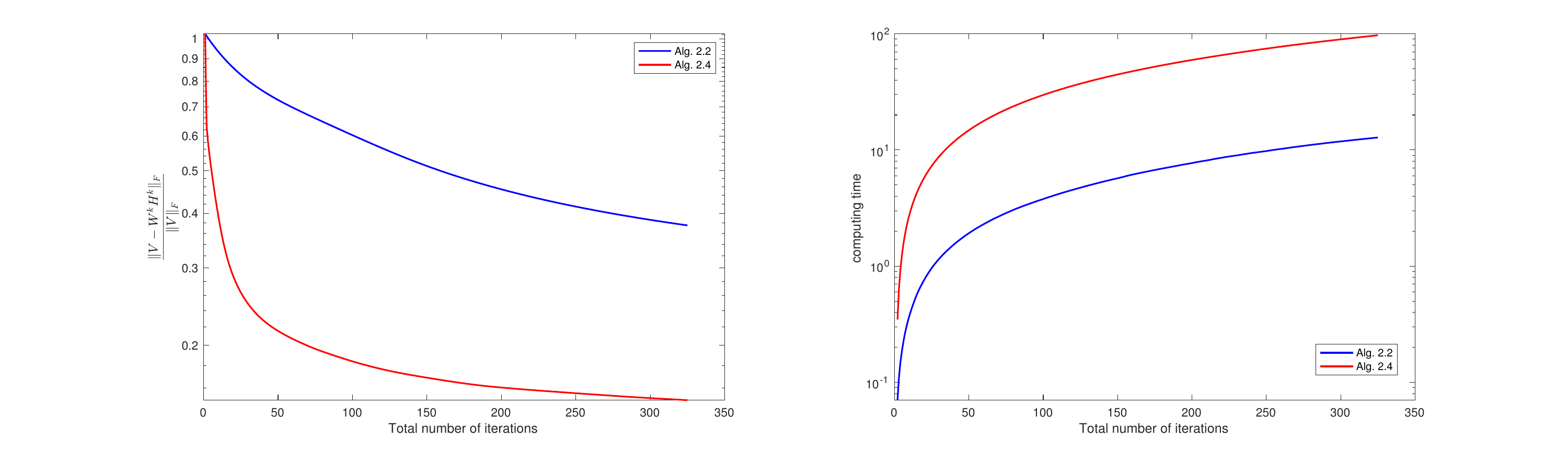}
 \caption{Convergence and computing time curves for one of the tests ($(s_1, s_2)=(150,120)$)  for Example \ref{ex:5}.} \label{fig4}
 \end{figure}

\begin{table}[!ht]\renewcommand{\arraystretch}{1.0} \addtolength{\tabcolsep}{1.0pt}
\begin{center}{\scriptsize
  \begin{tabular}[c]{|c|c|c|c|c|c|c|}     \hline
\multirow{2}{*}{$(s_1,s_2)$}& \multicolumn{3}{|c|}{Alg. \ref{PALM}}  & \multicolumn{3}{|c|}{Alg.  \ref{alg1}}  \\ \cline{2-7}
& {\tt res.}   & ${\tt H(V,W^kH^k)}$ & {\tt ct.} & {\tt res.} & ${\tt H(V,W^kH^k)}$& {\tt ct.} \\ \hline
 $(100, 100)$  & $12.129$ & $0.3132$ & $0.3753$  & $0.1368$   & $0.1293$           & $371.40$  \\ \hline
 $(150, 120)$  & $12.802$ & $0.3130$ & $0.3754$  & $0.1329$   & $0.1264$           & $403.47$  \\ \hline
 \end{tabular}}
\end{center}
\caption{Numerical results for Example \ref{ex:5}.}\label{table5}
\end{table}

Next, we consider a numerical example in  COVID-19 open Research Dataset \cite{WL20}.
\begin{example}\label{ex:6}
The CORD-19 data\footnote{\url{https://www.semanticscholar.org/cord19/download}}, offered by Allen Institute for AI and other leading research groups, is a growing resource containing all scientific papers on Covid-19 and related historical coronavirus research. We select one of the file packages and use the abstracts of the articles in the file package for the experiment. First, we remove the abstracts with a total number of words less than 400, then remove the meaningless words in the abstracts and remove the words that appear less than 20 times. Finally, we count the times of all words in each abstract document and normalized each column to have total sum of 1. Finally, we get a $12801 \times 8625$ target matrix $V$.
\end{example}

Figure \ref{covidfig} and Table \ref{covidtable} show the quality of recovery and the sparsity of the computed solution via Algorithms \ref{PALM} and \ref{alg1} with $(s_1,s_2)=(4500,5)$ and ARTMsparse via running 350 steps on the CORD-19 data matrix $V$.

We observe from Figure \ref{covidfig} and Table \ref{covidtable} that  Algorithm \ref{alg1} can achieve much smaller reconstruction error than the other methods while the computed solutions via  Algorithms \ref{PALM} and \ref{alg1} are  more sparse than the solution obtained by ARTMsparse.

\begin{figure}[!ht]
 \centering
\includegraphics[width=0.65\textwidth]{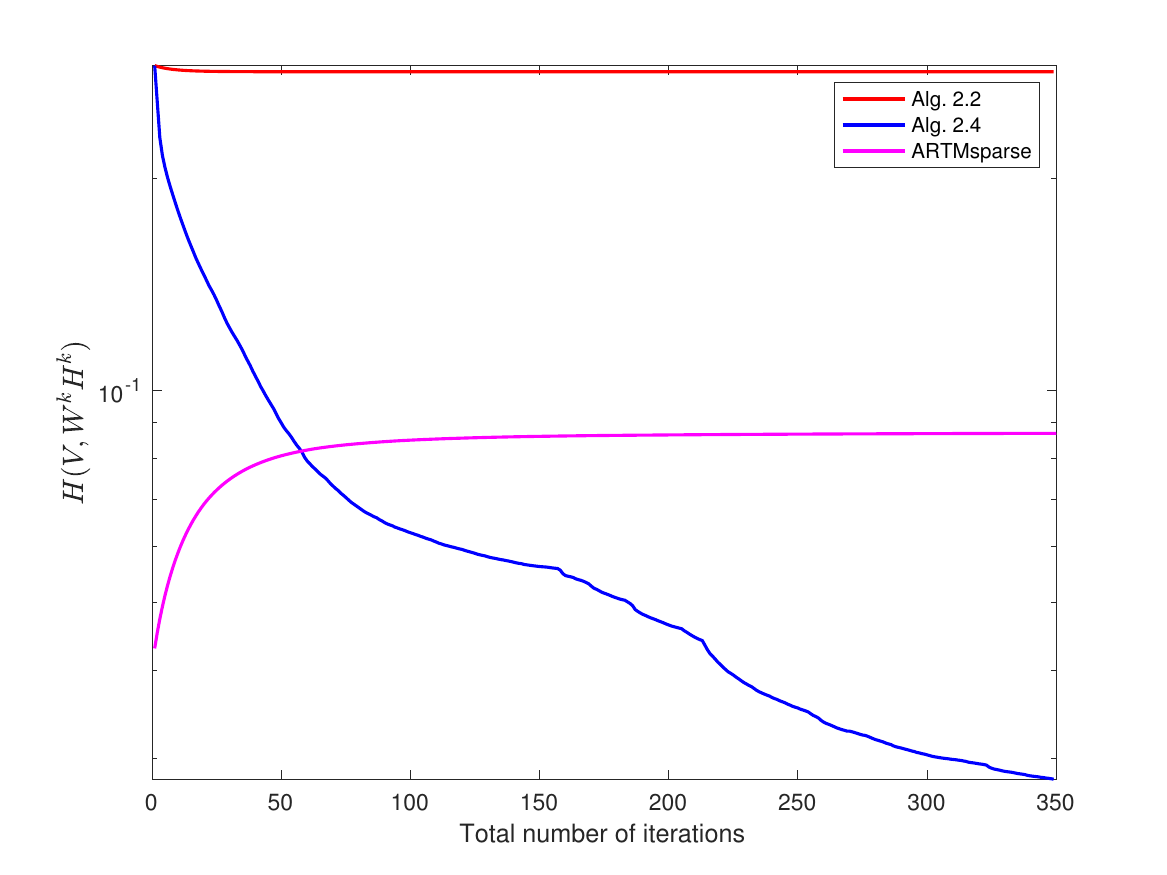}
 \caption{Reconstruction errors for Example \ref{ex:6}.}\label{covidfig}
 \end{figure}

\begin{table}[!ht]\renewcommand{\arraystretch}{1.0} \addtolength{\tabcolsep}{1.0pt}
\begin{center}{\scriptsize
  \begin{tabular}[c]{|c|c|c|c|c|c|}     \hline
& $\|W^k\|_0$      & $\|H^k\|_0$  & {\tt res.} & {\tt $H(V,W^kH^k)$}  & {\tt ct.} \\ \hline
Alg. \ref{PALM}    &$135000$ & $43125$ & $0.5807$ & $0.2833$ & $4.6933\times 10^{2}$  \\ \hline
Alg. \ref{alg1}    &$135000$ & $43125$ & $0.1162$ & $0.0280$ & $1.2355\times 10^{4}$  \\ \hline
ARTMsparse         &$384030$ & $50483$ & $0.2390$ & $0.0868$ & $1.5349\times 10^{5}$  \\ \hline
\end{tabular}}
\end{center}
\caption{Numerical results for Example \ref{ex:6}.} \label{covidtable}
\end{table}

\section{Concluding remarks}\label{sec5}

In this paper,  we have considered the sparse stochastic matrix factorization, which is rewritten as  an unconstrained nonconvex-nonsmooth minimization problem. Then a column-wise update algorithm  is proposed for solving the  minimization problem. The global convergence of the proposed algorithm is established. Numerical experiments on both synthetic and real data sets demonstrate the effectiveness of our algorithm.
The main drawback of the sparse stochastic matrix factorization is that the factors are not necessarily unique in general. How to guarantee the uniqueness of our factorization method is a challenging question. Another interesting question is how to choose the sparsity level for the sparse stochastic matrix factorization. Therefore, further studies are needed.

\section*{Acknowledgments}
We are very grateful to the editor and the referees for their valuable comments and
suggestions, which have considerably improved this paper.

\end{document}